\documentclass{article}
\usepackage[letterpaper, portrait, margin=2.5cm]{geometry}

\usepackage{verbatim}
\usepackage[utf8]{inputenc}
\usepackage{amsmath,amssymb,amsthm}
\usepackage{tikz-cd}
\usepackage{bbm}

\usepackage[blocks,affil-it]{authblk}

\usepackage{comment}
\usepackage{stmaryrd}
\usepackage[new]{old-arrows}
\usepackage{graphicx}
\graphicspath{ {./Images/} }
\usepackage{todonotes}

\usepackage{algorithm}
\usepackage{algpseudocode}

\usepackage{mathtools}

\usepackage [english]{babel}

\newcommand{\catname}[1]{{\normalfont\textbf{#1}}}
\newcommand{\Set}{\catname{Set}}

\newcommand{\AAA}{\mathbb A}
\newcommand{\CC}{\mathbb C}
\newcommand{\PP}{\mathbb P}

\newcommand{\VV}{\mathbb V}
\newcommand{\ZZ}{\mathbb Z}

\newcommand{\mcc}{\mathcal C}
\newcommand{\mct}{\mathcal T}

\newcommand{\mfg}{\mathfrak G}

\newcommand{\lbm}{\left[ \begin{matrix}}
\newcommand{\rem}{\end{matrix} \right]}
\newcommand{\lp}{\left(}
\newcommand{\rp}{\right)}
\newcommand{\lb}{\left\{}
\newcommand{\rb}{\right\}}
\newcommand{\lav}{\left|}
\newcommand{\rav}{\right|}

\newcommand{\SL}{\sum\limits}

\newcommand{\ra}{\rightarrow}

\newcommand{\injects}{\hookrightarrow}

\newcommand{\inv}{^{-1}}
\newcommand{\st}{ \ | \ }
 
\newcommand{\p}{\partial}

\DeclareMathOperator{\Hom}{Hom}

\DeclareMathOperator{\RD}{RD}
\DeclareMathOperator{\ed}{ed}

\newcommand{\In}{\textbf{in}}
\newcommand{\Range}{\textbf{range}}
\newcommand{\Len}{\textbf{len}}

\algrenewcommand\algorithmicdo{}

\newcommand{\red}{\text{red}}
\newcommand{\Syl}{\text{Syl}}

\theoremstyle{definition}
\newtheorem{definition}{Definition}[section]
\newtheorem{remark}[definition]{Remark}

\newtheorem{prop}[definition]{Proposition}
\newtheorem{lemma}[definition]{Lemma}
\newtheorem{theorem}[definition]{Theorem}
\newtheorem{corollary}[definition]{Corollary}

\newtheorem{algorithmname}[definition]{Algorithm}
\newtheorem{question}[definition]{Question}

\title{Upper Bounds on Resolvent Degree via\\ Sylvester's Obliteration Algorithm}
\author{Curtis Heberle and Alexander J. Sutherland\footnote{The second author was supported in part by the National Science Foundation under Grant No. DMS-1944862.}}
\date{}

\begin{document}

\maketitle

\begin{abstract}
For each $n$, let $\RD(n)$ denote the minimum $d$ for which there exists a formula for the general polynomial of degree $n$ in algebraic functions of at most $d$ variables. In this paper, we recover an algorithm of Sylvester for determining non-zero solutions of systems of homogeneous polynomials, which we present from a modern algebro-geometric perspective. We then use this geometric algorithm to determine improved thresholds for upper bounds on $\RD(n)$.
\end{abstract}

\tableofcontents

\section{Introduction}\label{sec:Introduction}
A classical problem in mathematics is to determine the roots of a general degree $n$ polynomial in one variable in terms of its coefficients. Modern work on this problem centers around resolvent degree, an invariant whose ideas permeate classical work, but was not formally defined until the independent definitions of Brauer \cite[p.46]{Brauer1975} and Arnol'd and Shimura \cite[p.46]{ArnoldShimura1976}. Farb and Wolfson greatly expanded the context of resolvent degree in \cite[Definition 2.3, Proposition 2.4, Definition 3.1]{FarbWolfson2019}. 

Following \cite[Example 4.2]{Wolfson2021}, we denote the resolvent degree of the general degree $n$ polynomial by $\RD(n)$. Currently, non-trivial lower bounds on $\RD(n)$ are unknown \cite[Section 1.5]{FarbWolfson2019}; it is possible that $\RD(n)=1$ for all $n$. Nonetheless, Dixmier \cite{Dixmier1993} noted that ``Every reduction of $\RD(n)$ would be serious progress,'' and Wolfson provided new upper bounds on $\RD(n)$ \cite[Theorems 5.6 and 5.8]{Wolfson2021} by constructing a ``bounding function'' $F(m)$ such that $\RD(n) \leq n-m$ for $n \geq F(m)$. The current best upper bounds on $\RD(n)$ are given by \cite[Theorem 3.27]{Sutherland2021C}, where the second-named author constructs an improved bounding function $G(m)$ and shows that $\lim\limits_{m \ra \infty} \frac{F(m)}{G(m)} = \infty$. 

In this paper, we recover an algorithm from \cite{Sylvester1887} (henceforth referred to as the ``obliteration algorithm'') for solving systems of equations using polynomials of minimal degree. An additional modern description of the Sylvester's work and its relevance to resolvent degree is given in \cite{Heberle2021}. Here we present the algorithm primarily from an algebro-geometric viewpoint using the language of ``polar cones'' introduced in \cite[Section 2]{Sutherland2021C}. We then use the obliteration algorithm to determine the following new upper bounds on resolvent degree:

\begin{theorem}\label{thm:Upper Bounds on Resolvent Degree} \textbf{(Upper Bounds on Resolvent Degree)}
\begin{enumerate}
	\item For $n \geq 5,250,199$, $\RD(n) \leq n-13$.
	\item For each $14 \leq m \leq 17$ and $n > \frac{(m-1)!}{120}$, $\RD(n) \leq n-m$.
	\item For $n \geq 381,918,437,071,508,901$, $\RD(n) \leq n-22$. 
	\item For each $23 \leq m \leq 25$ and $n > \frac{(m-1)!}{720}$, $\RD(n) \leq n-m$. 
\end{enumerate}

The above result is found as Theorem \ref{thm:Bounds from the Geometric Obliteration Algorithm} in Section \ref{sec:Upper Bounds on Resolvent Degree} and leads to the construction of a new bounding function $G'(m)$ such that $\RD(n) \leq n-m$ for $n \geq G'(m)$ and $G'(m) \leq G(m)$ in Corollary \ref{cor:The New Bounding Function}.

\end{theorem}

\paragraph{Historical Remarks} The second-named author uses two distinct methods to construct $G(m)$ \cite[Theorems 3.7, 3.10, 3.24]{Sutherland2021C}. For general $m$ (Theorem 3.24), the second-named author uses a result of Debarre and Manivel \cite[Theorem 2.1]{DebarreManivel1998} to improve on the construction of Wolfson which underlies \cite[Theorem 5.6]{Wolfson2021}. For small $m$ (Theorems 3.7 and 3.10), the second-named author uses iterated polar cone methods which build upon the methods of \cite{Wiman1927}, \cite{Chebotarev1954}, and \cite{Segre1945} (note, however, that Wiman and Chebotarev do not use the language of polars at all and Segre refers only to individual polars). An application of Sylvester's obliteration algorithm to certain small $m$ cases is considered in \cite{Heberle2021}. By combining Sylvester's obliteration algorithm with the other methods described above, the authors believe they have exhausted the classical methods related to the theory of Tschirnhaus transformations; implications of this are discussed in Subsection \ref{subsec:Remaining Questions}.

\paragraph{Outline of the Paper}
In Section 2, we recall the relevant background on resolvent degree, polar cones, and Tschirnhaus transformations. In Section 3, we present a modern, geometric version of the obliteration algorithm and related phemonena, as well as a summary of Sylvester's original work. In Section 4, we apply the geometric obliteration algorithm to obtain upper bounds on resolvent degree. In Section 5, we discuss Python implementations of the geometric obliteration algorithm used for computations relevant for Theorem \ref{thm:Bounds from the Geometric Obliteration Algorithm}.

\paragraph{Conventions}
\begin{enumerate}
	\item We restrict to fields $K$ which are finitely generated $\CC$-algebras. One could instead fix an arbitrary algebraically closed field $F$ of characteristic zero (in lieu of $\CC$) and the statements (relative to $F$) would hold.
	\item We follow the conventions of \cite{Harris2010} for algebraic varieties. In particular, a projective (respectively, affine) variety is defined to be a closed algebraic set in $\PP_K^r$ (respectively, $\AAA_K^r$). When we say variety without a specific modifier, we mean a quasi-projective variety. Note that we do not assume that varieties are irreducible.
	\item Given $a,b \in \ZZ_{\geq 0}$, we set $[a,b] = \lb x \in \ZZ \st a \leq x \leq b \rb$.
	\item Given a collection of homogeneous polynomials $S = \lb f_1,\dotsc,f_s \rb \subseteq K[x_0,\dotsc,x_r]$, we write $\VV(f_1,\dotsc,f_s)$ (and occasionally $\VV(S)$) for the subvariety of $\PP_K^r$ determined by the conditions $f_1 = \cdots = f_s = 0$.
	\item Given a subvariety $V \subseteq \PP_K^r$, we write $V(K)$ for the set of $K$-rational points of $V$.
	\item Given points $P_0,\dotsc,P_\ell \in \PP^r(K)$, we write $\Lambda(P_0,\dotsc,P_\ell)$ for the linear subvariety of $\PP_K^r$ that they determine. Additionally, we refer to a linear subvariety $\Lambda \subseteq \PP_k^r$ of dimension $k \geq 3$ as a $k$-plane. We refer to linear subvarieties of dimension 1 (respectively, 2) as lines (respectively, planes).
	\item We use the notation $K_n$ to mean $\CC(a_1,\dotsc,a_n)$, a purely transcendental extension of $\CC$ with transcendence basis $a_1,\dotsc,a_n$.  
\end{enumerate}

Note that for generic choices of $f_1,\dotsc,f_s$, $\VV(f_1,\dotsc,f_s)$ is a complete intersection. However, there are examples of such choices which are not complete intersections, such as the twisted cubic curve. Following the convention of \cite{Sutherland2021C}, we refer to a subvariety $\VV(f_1,\dotsc,f_s)$ as an intersection of hypersurfaces.

Consider a system of equations $S$ where each polynomial has degree at most $d$ and where we denote the number of polynomials of degree $j$ by $\ell_j$. In such a case, we say that $S$ is of type $\lbm d &\cdots &1\\ \ell_d &\cdots &\ell_1 \rem$. If $\ell_j=0$ for any $j \in [1,d-1]$, the corresponding column may be omitted from the presentation. When $d \geq 2$ and each $\ell_j=1$, we say $S$ is of type $(1,\dotsc,d)$.

When $V = \VV\lp f_1,\dotsc,f_s \rp$, we say that the type of $V$ is the type of the system $\lb f_1,\dotsc,f_s \rb$. We note that the type of $V$ explicitly depends on the presentation in terms of $f_1,\dotsc,f_s$; it is not unique. However, we only consider the type of an intersection of hypersurfaces when it is defined by an explicit set of polynomials.

\paragraph{Acknowledgements} The authors thank David Ishii Smyth and Jesse Wolfson for their support. The second author thanks Joshua Jordan for helpful conversations. Additionally, the authors thank the anonymous referee for many helpful comments and suggestions. 

\section{Resolvent Degree, Polar Cones, and Tschirnhaus Transformations}\label{sec:Resolvent Degree and Polar Cones}

\subsection{Resolvent Degree}\label{subsec:Resolvent Degree}

We refer the reader to \cite{FarbWolfson2019} for general definitions of resolvent degree (Definitions 1.3, 2.3), a summary of its history (Section 1), and additional context. We only work over $\CC$ and thus provide definitions in this context.  

\begin{definition}\label{def:Resolvent Degree for Fields} \textbf{(Resolvent Degree of Field Extensions)}\\
Let $K'/K$ be an extension of $\CC$-fields. The \textbf{resolvent degree} of $K'/K$, denoted $\RD\lp L/K \rp$, is the minimal $d$ for which there exists a tower of finite extensions
\begin{equation*}
	K = E_0 \injects E_1 \injects \cdots \injects E_\ell
\end{equation*}

\noindent
such that $K'$ embeds into $E_\ell$ over $K$ and the essential dimension of each $E_{j+1}/E_j$ is at most $d$. 
\end{definition}

\begin{definition}\label{def:Resolvent Degree for Maps} \textbf{(Resolvent Degree of Generically Finite, Dominant Maps)}\\
Let $Y \dashrightarrow X$ be a generically finite, dominant rational map of $\CC$-varieties. The \textbf{resolvent degree} of $Y \dashrightarrow X$, denoted $RD\lp Y \dashrightarrow X \rp$, is the minimal $d$ for which there exists a tower of generically finite, dominant rational maps
\begin{equation*}
	E_\ell \dashrightarrow \cdots \dashrightarrow E_1 \dashrightarrow E_0 = X
\end{equation*}

\noindent
such that $E_\ell \dashrightarrow X$ factors as $E_\ell \dashrightarrow Y \dashrightarrow X$ and the essential dimension of each $E_{j+1} \dashrightarrow E_j$ is at most $d$.
\end{definition}

We first note that Definitions \ref{def:Resolvent Degree for Fields} and \ref{def:Resolvent Degree for Maps} agree and is induced by sending an irreducible affine variety $X$ to the corresponding field of rational functions $\CC(X)$. We refer the reader to \cite[Definition 1.3]{FarbWolfson2019} for a precise definition of essential dimension, but note that we often use approximate essential dimension via the bounds $\ed\lp K'/K \rp \leq \text{tr.deg}\lp K \rp$ and $\ed\lp Y \dashrightarrow X \rp \leq \dim(X)$.

We write $\RD(n)$ for the resolvent degree of the general degree $n$ polynomial, which is given precisely as
\begin{align*}
	\RD(n) &= \RD\lp \CC^n \dashrightarrow \CC^n/S_n \rp,\\
	&= \RD\lp \CC\lp a_1,\dotsc,a_n \rp[z]/\lp z^n + a_1z^{n-1} + \cdots a_{n-1}z + a_n \rp / \CC\lp a_1,\dotsc,a_n \rp \rp.
\end{align*}

\noindent
Additionally, resolvent degree is defined for finite groups \cite[Definition 3.1]{FarbWolfson2019} and $\RD(n) = \RD\lp S_n \rp = \RD\lp A_n \rp$ \cite[Theorem 3.3, Corollary 3.17]{FarbWolfson2019}.

While we restrict ourselves to working over $\CC$, we do note lose any generality. Theorem 1.2 of \cite{Reichstein2022} yields that $\RD_\CC\lp S_n \rp = \RD_K\lp S_n \rp$ for any field $K$ of characteristic zero and \cite[Theorem 1.3]{Reichstein2022} yields that $\RD_\CC\lp S_n \rp \geq \RD_K\lp S_n \rp$ for any field $K$, i.e. resolvent degree can only go down in positive characteristic.

In Lemma \ref{lem:Properties of Resolvent Degree}, we summarize several basic results which will be used frequently (and often without explicit reference). Item 1 is the field-theoretic version of \cite[Lemma 2.7]{FarbWolfson2019} and follows immediately from the definition of resolvent degree. Items 2 and 3 are algebraic versions of \cite[Lemma 2.9]{FarbWolfson2019} and can be found explicitly as follows as \cite[Lemma 2.18, Proposition 2.19]{Sutherland2021C}. Note that items 2 and 3 follow directly from the primitive element theorem.

\begin{lemma}\label{lem:Properties of Resolvent Degree} \textbf{(Properties of Resolvent Degree)}
\begin{enumerate}
	\item Let $E_0 \injects E_1 \injects \cdots \injects E_\ell$ be a tower of field extensions. Then,
	\begin{align*}
		\RD(E_\ell/E_0) = \max\lb \RD(E_j/E_{j-1}) \st j \in [1,\ell] \rb.
	\end{align*}
	\item Let $K'/K$ be a degree $d$ field extension. Then, $\RD(K'/K) \leq \RD(d)$.
	\item Let $V \subseteq \PP_K^r$ be a degree $d$ subvariety. Then, there is an extension $K'/K$ with $\RD(K'/K) \leq \RD(d)$ over which we can determine a $K'$-rational point of $V$.
\end{enumerate}
\end{lemma}

\noindent
As a consequence of item 3, we say that we can determine a point of a degree $d$ subvariety $V$ by solving a degree $d$ polynomial.

\subsection{Polar Cones and $k$-Polar Points}\label{subsec:Polar Cones and k-Polar Points}

The original theory of polars for hypersurfaces is classical and a classical reference is \cite{Bertini1923}; a modern reference on polars is \cite{Dolgachev2012}. We now recall the key definitions and results of \cite[Section 2]{Sutherland2021C}; we use the same notation and begin with the definition of polars. 

\begin{definition}\label{def:Polars and Polar Cones} \textbf{(Polars and Polar Cones)} \\
Let $f \in K[x_0,\dotsc,x_r]$ be a homogeneous polynomial of degree $d$ and $P \in \PP^r(K)$. Observe that the set
\begin{equation*}
	I_j^* := \Hom_\Set\lp [1,j] , [0,r] \rp
\end{equation*}

\noindent
indexes the (ordered) $j^{th}$ partial derivatives of $f$ for each $j \in [0,d]$. We also use the shorthand
\begin{equation*}
	\p_0^{j_0} \cdots \p_\ell^{j_\ell} = \frac{\p^{j_0+\cdots+j_\ell}}{\p x_0^{j_0} \cdots \p x_\ell^{j_\ell}}.
\end{equation*}

\noindent
For each $j \in [0,d]$, the \textbf{$j^{th}$ polar of $f$ at $P$} is the homogeneous polynomial
\begin{equation}\label{eqn:jth polar of a polynomial}
	t(j,f,P) := \SL_{\iota \in I_{d-j}^*} \lp \p_0^{\lav \iota\inv(0) \rav} \cdots \p_r^{\lav \iota\inv(r) \rav} f \rp \biggr\rvert_{P}  x_0^{\lav \iota\inv(0) \rav} \cdots x_r^{\lav \iota\inv(r) \rav},
\end{equation}

\noindent
which is of degree $d-j$. Next, consider the hypersurface $H = \VV(f)$. The \textbf{$j^{th}$ polar of $H$ at $P$} is
\begin{align*}
	T(j,f,P) := \VV(t(j,f,P)) \subseteq \PP_K^r.
\end{align*}

\noindent
Finally, the \textbf{(first) polar cone of $H$ at $P$} is 
\begin{equation*}
	\mcc(H;P) := \bigcap\limits_{j=0}^{d-1} T(j,f,P). 
\end{equation*}
\end{definition}

Note that $T(0,f,P)=H$ for all $P$ and $T(d,f,P) = \PP_K^r$ if $P \in H(K)$. If $H$ is smooth at $P$, then $T(d-1,f,P)$ is the tangent hyperplane of $H$ at $P$. Our interest in polars stems from our interest in polar cones, which are themselves motivated by the following classical result (which is stated as a fact in \cite[I.5, p.292]{Segre1945}; Segre refers readers to \cite[p.203]{Bertini1923}).

\begin{lemma}\label{lem:Bertini's Lemma for Hypersurfaces} \textbf{(Bertini's Lemma for Hypersurfaces)} \\
Let $H \subseteq \PP_K^r$ be a hypersurface and $P \in H(K)$. Then, $\mcc(H;P) \subseteq H$ is a cone with vertex $P$. 
\end{lemma}

\noindent
In particular, for any point $Q \in \mcc(H;P) \setminus \{P\}$, the line $\Lambda(P,Q)$ lies in $H$.

Observe that for an intersection of hypersurfaces $\VV(f_1,\dotsc,f_s)$, a line $\Lambda$ lies on $\VV(f_1,\dotsc,f_s)$ exactly when $\Lambda$ lies on each hypersurface $\VV(f_j)$. This observation motivates the following definition and lemma, which are originally given as \cite[Definition 2.10, Lemma 2.11]{Sutherland2021C}.

\begin{definition}\label{def:Polar Cone of an Intersection of Hypersurfaces} \textbf{(Polar Cone of an Intersection of Hypersurfaces)} \\
Let $V = \VV(f_1,\dotsc,f_s) \subseteq \PP_K^r$ be an intersection of hypersurfaces and $P \in V(K)$. The \textbf{(first) polar cone of $V$ at $P$} is
\begin{equation*}
	\mcc(V;P) := \bigcap\limits_{j=1}^s \mcc(\VV(f_j);P). 
\end{equation*}
\end{definition}

\begin{lemma}\label{lem:Bertini's Lemma for Intersections of Hypersurfaces}  \textbf{(Bertini's Lemma for Intersections of Hypersurfaces)} \\
Let $V \subseteq \PP_K^r$ be an intersection of hypersurfaces and $P \in V(K)$. Then, $\mcc(V;P) \subseteq V$ is a cone with vertex $P$. 
\end{lemma}

Iterating the polar cone construction yields a method for determining $k$-planes on intersections of hypersurfaces. We now recall the associated definitions, first given as \cite[Definition 2.22]{Sutherland2021C}.

\begin{definition}\label{def:Iterated Polar Cones and k-Polar Points} \textbf{(Iterated Polar Cones and k-Polar Points)} \\
Let $V \subseteq \PP_K^r$ be an intersection of hypersurfaces and $P_0 \in V(K)$. First, set $\mcc^1(V;P_0) := \mcc(V;P_0)$. Given additional points $P_1,\dotsc,P_{k-1} \in V(K)$ such that
\begin{equation*}
	P_{\ell} \in \mcc^{\ell}(V;P_0,\dotsc,P_{\ell-1}) \setminus \Lambda\lp P_0,\dotsc,P_{\ell-1} \rp
\end{equation*}

\noindent
for $\ell \in [1,k-1]$, the \textbf{$k^{th}$ polar cone of $V$ at $P_0,\dotsc,P_{k-1}$} is
\begin{equation*}
	\mcc^k(V;P_0,\dotsc,P_{k-1}) := \mcc\lp \mcc^{k-1}(V;P_0,\dotsc,P_{k-2});P_{k-1} \rp.
\end{equation*}

\noindent
We refer to an ordered collection of such points $(P_0,\dotsc,P_k)$ as a \textbf{$k$-polar point} of $V$. 

If the points $P_0,\dotsc,P_{k-1}$ have already been chosen, we refer to $\mcc^k(V;P_0,\dotsc,P_{k-1})$ as \emph{the} $k^{th}$ polar cone of $V$. In the event that such points exist, but have not been explicitly chosen, we refer to \emph{a} $k^{th}$ polar cone of $V$. Additionally, it is sometimes useful to refer to $V$ itself as a zeroth polar cone of $V$ (at any of its $K$-points). 
\end{definition}

By noting that iterated polar cones are nested, i.e.
\begin{equation*}
	\mcc^k(V;P_0,\dotsc,P_{k-1}) \subseteq \mcc^{k-1}(V;P_0,\dotsc,P_{k-2}) \subseteq \cdots \subseteq \mcc^2(V;P_0,P_1) \subseteq \mcc(V;P_0) \subseteq V
\end{equation*}

\noindent
and that the points $P_0,\dotsc,P_k$ defining a $k$-polar point $(P_0,\dotsc,P_k)$ are in general position, we arrive at the following $k$-plane analogue of Lemma \ref{lem:Bertini's Lemma for Intersections of Hypersurfaces}, which is originally \cite[Lemma 2.24]{Sutherland2021C}:

\begin{lemma}\label{lem:Polar Point Lemma} \textbf{(Polar Point Lemma)} \\
Let $V \subseteq \PP_K^r$ be an intersection of hypersurfaces and let $(P_0,\dotsc,P_k)$ be a $k$-polar point of $V$. Then, $\Lambda(P_0,\dotsc,P_k) \subseteq \mcc^k(V;P_0,\dotsc,P_{k-1}) \subseteq V$ is a $k$-plane.
\end{lemma}

\subsection{Tschirnhaus Transformations}

We use the notation and conventions of \cite[Subsection 3.1]{Sutherland2021C} for Tschirnhaus transformations and refer the reader there for details. Note also that Wolfson provides a more complete history of Tschirnhaus transformations in \cite[Section 2 and Appendix B]{Wolfson2021}. Let $K_n = \CC(a_1,\dotsc,a_n)$ be a purely transcendental extension of $\CC$ with transcendence basis $a_1,\dotsc,a_n$. 

\begin{definition}\label{def:General Polynomials} \textbf{(General Polynomials)} \\
The \textbf{general polynomial of degree $n$} is the polynomial
\begin{equation*}
	\phi_n(z) = z^n + a_1z^{n-1} + \cdots + a_{n-1}z + a_n \in K_n[z].
\end{equation*}
\end{definition}

\begin{definition}\label{def:Tschirnhaus Transformations} \textbf{(Tschirnhaus Transformations)} \\
A \textbf{Tschirnhaus transformation} of the general degree $n$ polynomial is an isomorphism of $K_n$-fields
\begin{equation*}
	\Upsilon:K_n[z]/(\phi_n(z)) \ra K_n[z]/(\psi(z)),
\end{equation*}

\noindent
where $\psi(z) = z^n + b_1z^{n-1} + \cdots + b_{n-1}z + b_n$. We say that $\Upsilon$ has type $(j_1,\dotsc,j_k)$ if $b_{j_1} = \cdots = b_{j_k} = 0$.
\end{definition}

As per Remark 3.3 of \cite{Sutherland2021C}, the space of all Tschirnhaus transformations of the general degree $n$ polynomial (up to re-scaling) is
\begin{equation*}
	\mct_{K_n}^n := \PP_{K_n}^{n-1} \setminus [1:0:\cdots:0] \subseteq \PP_{K_n}^{n-1}. 
\end{equation*}

\noindent
Note that each $b_j$ in Definition \ref{def:Tschirnhaus Transformations} is a homogeneous polynomial of degree $j$ in $a_1,\dotsc,a_n$.

\begin{definition}\label{def:Tschirnhaus Complete Intersections} \textbf{(Tschirnhaus Complete Intersections)} \\
Fix $n \in \ZZ_{\geq 1}$. For any $m \in [1,n-1]$, the \textbf{$m^{th}$ extended Tschirnhaus hypersurface is}
\begin{equation*}
	\tau_m := \VV(b_m) \subseteq \PP_{K_n}^{n-1},
\end{equation*}

\noindent
and the \textbf{$m^{th}$ extended Tschirnhaus complete intersection} is
\begin{equation*}
	\tau_{1,\dotsc,m} := \bigcap\limits_{j=1}^m \tau_j \subseteq \PP_{K_n}^{n-1}.
\end{equation*}

\noindent
Additionally, the \textbf{$m^{th}$ Tschirnhaus hypersurface} is
\begin{equation*}
	\tau_m^\circ := \tau_m \cap \mct_{K_n}^n = \tau_m \setminus \lb [1:0:\cdots:0] \rb,
\end{equation*} 

\noindent
and the \textbf{$m^{th}$ Tscihrnhaus complete intersection is}
\begin{equation*}
	\tau_{1,\dotsc,m}^\circ := \tau_{1,\dotsc,m} \cap \mct_{K_n}^n = \tau_{1,\dotsc,m} \setminus \lb [1:0:\cdots:0] \rb.
\end{equation*}
\end{definition}

\begin{remark}\label{rem:Strategy for Upper Bounds on RD(n)} \textbf{(Strategy for Upper Bounds on $\RD(n)$}\\
If we can determine a $K'$-rational point of $\tau_{1,\dotsc,m-1}^\circ$ over an extension $K'/K_n$ of sufficiently small resolvent degree, then we can conclude that $\RD(n) \leq n-m$. Notice that if we can determine an $(m-d-1)$-plane $\Lambda \subseteq \tau_{1,\dotsc,d}^\circ$ over an extension $L/K_n$ of low resolvent degree, then we need only further pass to an extension $K'/L$ with $\RD(K'/L) \leq \RD\lp \frac{(m-1)!}{d!} \rp$, by Lemma \ref{lem:Properties of Resolvent Degree}. 
\end{remark}

Lemma \ref{lem:Polar Point Lemma} yields that every $k$-polar point determines a $k$-plane, hence Remark \ref{rem:Strategy for Upper Bounds on RD(n)} yields that it will suffice to determine $k$-polar points on the Tschirnhaus complete intersections $\tau_{1,\dotsc,d}^\circ$.

\section{The Obliteration Algorithms}\label{sec:The Obliteration Algorithms}

In \cite{Sylvester1887}, Sylvester gives an algorithm to determine an upper bound on the number of variables required to determine a non-trivial solution for a system of homogeneous polynomials of given degrees by solving polynomials of the same, or lower, degrees. The algorithm centers on Sylvester's ``formula of obliteration'' \cite[p.475]{Sylvester1887}, which will be covered in detail in Corollary \ref{cor:Geometric Formula of Obliteration} and Proposition \ref{prop:Sylvester's Formula of Obliteration}. Consequently, we refer to Sylvester's method as the ``obliteration algorithm.'' In Subsection \ref{subsec:The Geometric Obliteration Algorithm}, we give a modern description of the obliteration algorithm via geometry (in terms of varieties, rational points, and polar cones). In Subsection \ref{subsec:Sylvester's Obliteration Algorithm}, we describe the obliteration algorithm in terms of systems of homogeneous polynomials and explain Sylvester's classical language.

\subsection{The Geometric Obliteration Algorithm}\label{subsec:The Geometric Obliteration Algorithm}
WE now give a geometric construction of Sylvester's obliteration algorithm. More specifically, given an intersection of hypersurfaces $V \subseteq \PP_K^r$, we give a bound on the ambient dimension required to be able to determine a point of $V$ over an extension $K'/K$ of bounded resolvent degree. Note that this bound depends only on the type of $V$.

\begin{definition}\label{def:Minimal Dimension Bound} \textbf{(Minimal Dimension Bound)}\\
The \textbf{minimal dimension bound of type $\lbm d &\cdots &1\\ \ell_d &\cdots &\ell_1 \rem$}, denoted $r(d;\ell_d,\dotsc,\ell_1)$ is the minimal $r' \in \ZZ_{\geq 1} \cup \{\infty\}$ such that whenever $r \geq r'$, we can determine a point of any intersection of hypersurfaces of type $\lbm d &\cdots &1\\ \ell_d &\cdots &\ell_1 \rem$ in $\PP_K^r$ over an extension $K'/K$ with $\RD\lp K'/K \rp \leq \RD(d)$. Given an intersection of hypersurfaces $V$ of type $\lbm d &\cdots &1\\ \ell_d &\cdots &\ell_1 \rem$, we set $r(V) := r(d;\ell_d,\dotsc,\ell_1)$.
\end{definition}

\begin{remark}\label{rem:Finiteness of the Minimal Dimension Bound} \textbf{(Finiteness of the Minimal Dimension Bound)}\\
The main goal of this section is to establish an upper bound on $r(d;\ell_d,\dotsc,\ell_1)$. More specifically, we introduce a recursive, combinatorial bound $g(d;\ell_d,\dotsc,\ell_1)$ in Definition \ref{def:Geometric Dimension Bound} which we will show satisfies
\begin{equation}\label{eqn:Fundamental Inequality}
	r(d;\ell_d,\dotsc,\ell_1) \leq g(d;\ell_d,\dotsc,\ell_1).
\end{equation}

\noindent
The proof of inequality (\ref{eqn:Fundamental Inequality}) is exactly the geometric version of the obliteration algorithm. 
\end{remark}

We now give Definition \ref{def:Geometric Dimension Bound} and note that the underlying geometric intuition is explained in Lemma \ref{lem:The Reduction Lemma} and Remark \ref{rem:Geometric Insight for the Reduction Lemma}. 

\begin{definition}\label{def:Geometric Dimension Bound}  \textbf{(Geometric Dimension Bound)}\\
The \textbf{geometric dimension bound of type $\lbm 1 \\ \ell_1 \rem$} is $g(1;\ell_1) := \ell_1$. Similarly, the geometric dimension bound \textbf{of type $\lbm 2 &1\\ 1 &\ell_1 \rem$} is $g(2;1,\ell_1) := 1+\ell_1$ and the geometric dimension bound \textbf{of type $\lbm 2 &1\\ \ell_2 &\ell_1 \rem$} with $\ell_2 \geq 2$ is
\begin{equation*}
	g(2;\ell_2,\ell_1) := g(2;\ell_2-1,\ell_2+\ell_1+1).
\end{equation*}

\noindent
For $d \geq 3$, the geometric dimension bound \textbf{of type $\lbm d &d-1 &\cdots &2 &1\\ 1 &\ell_{d-1} &\cdots &\ell_2 &\ell_1 \rem$} is
\begin{equation*}
	g(d;1,\ell_{d-1},\dotsc,\ell_2,\ell_1) := g\lp d-1;\ell_{d-1},(\ell_{d-1}+\ell_{d-2}),\dotsc,\SL_{j=2}^{d-1} \ell_j, \lp \SL_{j=1}^{d-1} \ell_j \rp + 1 \rp.
\end{equation*}

\noindent
For $d \geq 3$ and $\ell_d \geq 2$, the geometric dimension bound \textbf{of type $\lbm d &d-1 &\cdots &2 &1\\ 1 &\ell_{d-1} &\cdots &\ell_2 &\ell_1 \rem$} is
\begin{equation*}
	g(d;\ell_d,\ell_{d-1},\dotsc,\ell_2,\ell_1) := g\lp d;\ell_d-1,(\ell_d+\ell_{d-1})-1,\dotsc,\lp \SL_{j=2}^d \ell_j \rp - 1, \SL_{j=1}^{d-1} \ell_j \rp.
\end{equation*}

\noindent
Finally, given an intersection of hypersurfaces $V$ of type $\lbm d &d-1 &\cdots &2 &1\\ \ell_d &\ell_{d-1} &\cdots &\ell_2 &\ell_1 \rem$, we set
\begin{equation*}
	g(V) := g(d;\ell_d,\dotsc,\ell_1).
\end{equation*}

\end{definition}

\begin{remark}\label{rem:Hyperplane Identities} \textbf{(Hyperplane Identities)}\\
The definitions of both the minimal and geometric dimension bounds admit a ``hyperplane identity,'' which we use without explicit reference:
	\begin{align*}\label{eqn:Hyperplane Identity}
		1+r(d;\ell_d,\dotsc,\ell_2,\ell_1) &= r(d;\ell_d,\dotsc,\ell_2,\ell_1+1),\\
		1+g(d;\ell_d,\dotsc,\ell_2,\ell_1) &= g(d;\ell_d,\dotsc,\ell_2,\ell_1+1).
	\end{align*}
\end{remark}
	
We next state Lemma \ref{lem:The Reduction Lemma}, which is the technical underpinning of the geometric obliteration algorithm and which specializes to give the geometric version of Sylvester's formula of reduction. 

\begin{lemma}\label{lem:The Reduction Lemma}  \textbf{(The Reduction Lemma)}\\
	Let $V$ be an intersection of hypersurfaces of type $\lbm d &\cdots &1\\ \ell_d &\cdots &\ell_1 \rem$ with $d \geq 2$ and which is not a hypersurface. Take $V_d$ to be a degree $d$ hypersurface and $V^\red$ to be an intersection of hypersurfaces of type 
\begin{equation*}
	\lbm d &\cdots &1\\ \ell_d-1 &\cdots &\ell_1 \rem
\end{equation*}

\noindent
if $\ell_d \geq 2$ and of type
\begin{equation*}
	\lbm d-1 &\cdots &1\\ \ell_{d-1} &\cdots &\ell_1 \rem
\end{equation*}

\noindent
	
\noindent
if $\ell_d=1$, such that $V = V^\red \cap V_d$. Let $P \in V^\red(K)$ and take $H$ to be a hyperplane which does not contain $P$. Then,
	\begin{equation*}
		g(V) = g(H \cap \mcc(V^\red;P)) = g(\mcc(V^\red;P))+1. 
	\end{equation*}
\end{lemma}

\begin{proof} 
	First, consider when $\ell_d \geq 2$. From Definition \ref{def:Polar Cone of an Intersection of Hypersurfaces}, observe that $\mcc(V^\red;P)$ has type
\begin{equation*}
	\lbm d &\cdots &1\\ \ell_{d}-1 &\cdots &\lp \SL_{j=1}^d \ell_j \rp -1\rem.
\end{equation*}.

\noindent
From Definition \ref{def:Geometric Dimension Bound}, it follows that
\begin{align*}
	g(V) &= g(d;\ell_d,\dotsc,\ell_1)\\
	&= g\lp d;\ell_d-1,(\ell_d+\ell_{d-1})-1,\dotsc,\lp \SL_{j=2}^d \ell_j \rp -1, \SL_{j=1}^d \ell_j \rp\\
	&= g\lp \mcc(V^\red;P) \rp + 1\\
	&=  g\lp H \cap \mcc(V^\red;P) \rp.
\end{align*}

\noindent
Similarly, when $\ell_d=1$, we have
\begin{align*}
	g(V) &= g(d;\ell_d,\dotsc,\ell_1)\\
	&= g\lp d-1;\ell_{d-1},(\ell_d+\ell_{d-1}),\dotsc, \SL_{j=2}^d \ell_j, \lp \SL_{j=1}^d \ell_j \rp + 1 \rp\\
	&= g\lp \mcc(V^\red;P) \rp + 1\\
	&=  g\lp H \cap \mcc(V^\red;P) \rp.
	\qedhere
\end{align*}
\end{proof}

\begin{remark}\label{rem:Geometric Insight for the Reduction Lemma} \textbf{(Geometric Insight for the Reduction Lemma)}\\
The proof of Lemma \ref{lem:The Reduction Lemma} follows immediately from Definition \ref{def:Geometric Dimension Bound}, but we wish to address the geometric reasoning underlying the lemma. Suppose our goal is to determine a point $Q$ of $V$ over an extension of bounded resolvent degree. Observe that if we can determine a line $\Lambda \subseteq V^\red$, then we need only solve a degree $d$ polynomial to determine a point of $V$. As $V^\red$ is $V$ with $V_d$ removed, it is already ``less difficult'' to determine the point $P \in V^\red(K)$ given by assumption (i.e. $g(V) \geq g(V^\red)$).  Additionally, we can determine a line $\Lambda \subseteq V^\red$ by determining a point $P' \not= P$ of $\mcc(V^\red;P)$. As $H$ is taken to be a hyperplane which does not contain $P$, it suffices to determine any point of $\mcc(V^\red;P) \cap H$, which is also ``less difficult'' as $\mcc(V^\red;P)$ is defined by fewer top degree hypersurfaces.
\end{remark}

As in Lemma \ref{lem:The Reduction Lemma}, we will frequently want to split an intersection of hypersurfaces $V$ into parts analogous to $V^\red$ and $V_d$, and so we introduce the following terminology and notation.

\begin{definition}\label{def:Reduction and Complement} \textbf{(Reduction and Complement)}\\
	Given an intersection of hypersurfaces $V$ of type $\lbm d &\cdots &1\\ \ell_d &\cdots &\ell_1 \rem$ with $\ell_d \geq 2$, a \textbf{reduction} of $V$ is an intersection of hypersurfaces $V^{\red}$ of type $\lbm d &d-1 &\cdots &2 &1\\ \ell_d-1 &\ell_{d-1} &\cdots &\ell_2 &\ell_1 \rem$ such that $V = V^\red \cap V_d$ for some degree $d$ hypersurface $V_d$, we which refer to as a \textbf{complement of $V^\red$ for $V$}.
	
	When $V$ is an intersection of hypersurfaces of type $\lbm d &\cdots &1\\ \ell_d &\cdots &\ell_1 \rem$ with $\ell_d = 1$, a reduction of $V$ is an intersection of hypersurfaces $V^{\red}$ of type $\lbm d-1 &\cdots &1\\ \ell_{d-1} &\cdots &\ell_1 \rem$ such that $V = V^\red \cap V_d$	for some degree $d$ hypersurface $V_d$, we which refer to as a complement of $V^\red$ for $V$.
\end{definition}

With Lemma \ref{lem:The Reduction Lemma} and Definition \ref{def:Reduction and Complement} in place, we now state the geometric version of Sylvester's ``formula of reduction'' \cite[p.475]{Sylvester1887}.

\begin{corollary}\label{cor:Geometric Formula of Reduction}  \textbf{(Geometric Formula of Reduction)}\\
Let $W$ be an intersection of hypersurfaces of type $\lbm d &\cdots &1\\ \ell_d &\cdots &\ell_1 \rem$. Then, for any $P_0 \in W(K)$, any reduction $\mcc(W;P_0)^\red$, and any $P_1 \in \mcc(W;P_0)^\red(K)$, we have
\begin{align*}
	g\lp \mcc(W;P_0) \rp &= g\lp \mcc\lp \mcc(W;P_0)^\red;P_1 \rp \rp + 1.
\end{align*}
\end{corollary}

\begin{proof}
	This follows immediately as a special case of Lemma \ref{lem:The Reduction Lemma} applied to $V=\mcc(W;P_0)$.
\end{proof}

We will soon want to successively iterate Lemma \ref{lem:The Reduction Lemma} so that we can eliminate the hypersurfaces of largest degree from any intersection of hypersurfaces by introducing many hypersurfaces of strictly lower degree. This is achieved in Proposition \ref{prop:The Obliteration Proposition}. However, we first introduce additional language and notation to refer to the varieties which arise in this process of reduction.

\begin{definition}\label{def:Sylvester Reductions} \textbf{(Sylvester Reductions)}\\
	Let $V$ be an intersection of hypersurfaces of type $\lbm d &\cdots &1\\ \ell_d &\cdots &\ell_1 \rem$ with $d \geq 2$ and which is not a hypersurface. A \textbf{first partial Sylvester reduction of $V$} is
	\begin{equation*}
		V^\Syl(d;1) := \mcc(V^\red;P_0),
	\end{equation*} 
	
	\noindent
	where $V^\red$ is any reduction of $V$ and $P_0 \in V^\red(K)$. Proceeding inductively, for any $j \in [2,\ell_d]$, a \textbf{$j^{th}$ partial Sylvester reduction of $V$} is
	\begin{equation*}
		V^\Syl(d;j) := \mcc(H_{j-1} \cap V_{j-1}^\Syl;P_k) = H_{j-1} \cap \mcc(V_{j-1}^\Syl;P_k) ,
	\end{equation*}
	
	\noindent
	where $H_{j-1}$ is a hyperplane which does not contain $P_{j-1}$ and $P_j \in \lp H_{k-1} \cap V_{k-1}^\Syl(d;j-1) \rp(K)$.
	
	When $d \geq 3$, a \textbf{first Sylvester reduction of $V$} is
	\begin{equation*}
		V_1^\Syl := V^\Syl(d;\ell_d). 
	\end{equation*}
	
	\noindent
	For each $j \in [2,d-1]$, let $\lambda_{d-j+1}$ be the number of degree $d-j+1$ hypersurfaces defining a $(j-1)^{st}$ Sylvester reduction $V_{j-1}^\Syl$. Then, a \textbf{$j^{th}$ Sylvester reduction of $V$} is
	\begin{equation*}
		V_j^\Syl := \lp V_{j-1}^\Syl \rp^\Syl(d-j+1;\lambda_{d-j+1}).
	\end{equation*}
\end{definition}

Continuing with the notation of Definition \ref{def:Sylvester Reductions}, note that $V_j^\Syl$ is a variety obtained by repeatedly applying Lemma \ref{lem:The Reduction Lemma} to $V$ to remove all hypersurfaces of degree $> d-j$.

\begin{prop}\label{prop:The Obliteration Proposition} \textbf{(The Obliteration Proposition)}\\
Let $V$ be an intersection of hypersurfaces of type $\lbm d &\cdots &1\\ \ell_d &\cdots &\ell_1 \rem$ with $d \geq 2$ which is not a hypersurface. Then,
\begin{equation*}
	g(V) = g\lp V_1^\Syl \rp.
\end{equation*}

\noindent
for any first Sylvester reduction $V_1^\Syl$ of $V$. 
\end{prop}

\begin{proof}
	From Lemma \ref{lem:The Reduction Lemma} and Definition \ref{def:Sylvester Reductions}, it follows immediately that
	\begin{equation*}
		g\lp V^\Syl(d;j) \rp = g\lp V^\Syl(d;j+1) \rp
	\end{equation*}
	
	\noindent
	for each $j \in [1,\ell_d-1]$. Consequently, applying Lemma \ref{lem:The Reduction Lemma} to $V$ and its partial Sylvester reductions yields
	\begin{equation*}
		g(V) = g\lp V^\Syl(d;1) \rp = \cdots = g\lp V^\Syl(d;\ell_d-1) \rp = g\lp V^\Syl(d;\ell_d) \rp = g\lp V_1^\Syl \rp.
		\qedhere
	\end{equation*}
\end{proof}

\begin{remark}\label{rem:Geometric Dimension Bound via Obliteration}  \textbf{(Geometric Dimension Bound via Obliteration)}\\
	From the definition of the $j^{th}$ Sylvester reductions, we can iteratively apply Proposition \ref{prop:The Obliteration Proposition} to observe that
	\begin{align*}
		g(V) = g\lp V_1^\Syl \rp = \cdots = g\lp V_{d-2}^\Syl \rp = g\lp V_{d-1}^\Syl \rp,
	\end{align*}
	
	\noindent
	which provides the most succinct description of the central argument of the geometric obliteration algorithm.
\end{remark}

We now arrive at the geometric version of Sylvester's ``formula of obliteration'' as a specialization of Proposition \ref{prop:The Obliteration Proposition}. 

\begin{corollary}\label{cor:Geometric Formula of Obliteration} \textbf{(Geometric Formula of Obliteration)}\\
	Let $W$ be an intersection of hypersurfaces of type $\lbm d &\cdots &1\\ \ell_d &\cdots &\ell_1 \rem$  with $d \geq 2$. For any $P_0 \in W(K)$ and any Sylvester reduction $\mcc(W;P_0)_1^\Syl$, we have
	\begin{equation}\label{eqn:Geometric Formula of Obliteration}
		g( \mcc(W;P_0) ) = g\lp \mcc(W;P_0)_1^\Syl \rp. 
	\end{equation}
\end{corollary}

\begin{proof}
	This follows immediately as a special case of Proposition \ref{prop:The Obliteration Proposition} with $V = \mcc(W;P_0)$. 
\end{proof}

\begin{remark}\label{rem:Explicit Numerics of the Formula of Obliteration} \textbf{(Explicit Numerics of the Formula of Obliteration)}\\
	Sylvester's formula of obliteration \cite[p.475]{Sylvester1887}, which we address in Proposition \ref{prop:Sylvester's Formula of Obliteration}, is given numerically and, for notational reasons, he chooses to write the  statement in terms of ``linear solutions'' of $\mcc(W;P_0)^\Syl(d;\ell_d-1)$ instead of $g\lp \mcc(W;P_0)_1^\Syl \rp$. For this reason, we delay the discussion of numerics of the formula of obliteration to Subsection \ref{subsec:Sylvester's Obliteration Algorithm}.
\end{remark}

As we have established the reduction lemma and the obliteration proposition, which we used to recover Sylvester's formula of reduction and formula of obliteration, we proceed to prove inequality (\ref{eqn:Fundamental Inequality}). 

\begin{prop} \label{prop:Minimal vs. Geometric Dimension Bound} \textbf{(Minimal vs. Geometric Dimension Bound)}\\
	For every type $\lbm d &\cdots &1\\ \ell_d &\cdots &\ell_1 \rem$ of an intersection of hypersurfaces, $r(d;\ell_d,\dotsc,\ell_1) \leq g(d;\ell_d,\dotsc,\ell_1) < \infty$.
\end{prop}

\begin{proof} \textbf{(The Geometric Obliteration Algorithm)}\\
	We proceed by induction on $d$. First, observe that when $d=1$, it is immediate that
	\begin{equation*}
		r(1;\ell_1) = \ell_1 = g(1;\ell_1). 
	\end{equation*} 

	\noindent	
	We additionally consider the case $d=2$ before considering the general case. For the $d=2$ case, we proceed via induction on $\ell_2$. When $\ell_2=1$, $\deg(V)=2$ and thus we can determine a point of $V$ by solving a quadratic polynomial when
	\begin{align*}
		\dim\lp V \rp \geq r - (\ell_1+1) = 0.
	\end{align*}
	
	\noindent
	It follows that
	\begin{align*}
		r(2;1,\ell_1) = \ell_1+1 = g(2;1,\ell_1).
	\end{align*}		
	
	\noindent
	 Now, consider the case where $\ell_2 \geq 2$ is arbitrary. Our inductive hypothesis yields 
\begin{equation*}
r(2;\ell_2-1,\lambda_1) \leq g(2;\ell_2-1,\lambda_1),
\end{equation*}	 

\noindent
for any $\lambda_1 \geq 0$. Let $V^\red$ be a reduction of $V$ with complement $V_2$. As $V^\red$ is of type $\lbm 2 &1\\ \ell_2-1 &\ell_1 \rem$, we can determine a point $P_0$ of $V^\red$ over an iterated quadratic extension whenever $r \geq g(V^\red)$. Let $H$ be a hypersurface which does not contain $P_0$. Note that $H \cap \mcc(V^\red;P_0)$ is of type $\lbm 2 &1\\ \ell_2-1 &\ell_2+\ell_1 \rem$ and so we can similarly determine a point $P_1$ of $H \cap \mcc(V^\red;P_0)$ over an iterated quadratic extension whenever $r \geq g(V^\red)+1$. From Lemma \ref{lem:Bertini's Lemma for Intersections of Hypersurfaces}, we have that
\begin{equation*}
	 \Lambda(P_0,P_1) \subseteq \mcc(V^\red;P_0) \subseteq V^\red. 
\end{equation*}

\noindent
Thus, we can determine a point of $\Lambda(P_0,P_1) \cap V_2 \subseteq V$ over an additional quadratic extension. From Lemma \ref{lem:The Reduction Lemma}, it follows that
\begin{equation*}
	r(2;\ell_2,\ell_1) \leq \max\lb g(2;\ell_2-1,\ell_1), g(2;\ell_2-1,\ell_1+\ell_2) \rb = g(2;\ell_2-1,\ell_1+\ell_2) = g(2;\ell_2,\ell_1).
\end{equation*}

Now, let us return to our induction on $d$ and consider the case of general $d \geq 2$. Our inductive hypothesis for $d$ yields that $r(d-1;\lambda_{d-1},\dotsc,\lambda_1) \leq g(d-1;\lambda_{d-1},\dotsc,\lambda_1)$ for any $\lambda_{d-1} \geq 1$ and $\lambda_j \geq 0$ for all $j \in [1,d-2]$. We proceed by induction on $\ell_d$. Let $V^\red$ be a reduction of $V$ with complement $V_d$. When $\ell_d=1$, the inductive hypothesis on $d$ yields that we can determine a point $P_0$ of $V^\red$ by solving polynomials of degree at most $d-1$ when $r \geq g(V^\red)$. Letting $H$ denote a hyperplane which does not contain $P_0$, we can similarly determine a point $P_1$ of $H \cap \mcc(V^\red;P_0)$ over by solving polynomials of degree at most $d-1$ when $r \geq g\lp \mcc(V^\red;P_0) \rp + 1$. It follows that
\begin{equation*}
	\Lambda(P_0,P_1) \subseteq \mcc(V^\red;P_0) \subseteq V^\red,
\end{equation*}

\noindent
and so we can determine a point of $\Lambda(P_0,P_1) \cap V_d \subseteq V$ by solving a degree $d$ polynomial. As a result,
\begin{equation*}
	r(d;1,\ell_{d-1},\dotsc,\ell_1) \leq \max\lb g(V^\red), g\lp \mcc(V^\red;P_0) \rp + 1 \rb = g\lp \mcc(V^\red;P_0) \rp + 1 = g(V) = g(d;1,\ell_{d-1},\dotsc,\ell_1).
\end{equation*}

Next, we consider the case of arbitrary $\ell_d \geq 2$. Our inductive hypothesis for $\ell_d$ yields that
\begin{equation*}
	r(d;\ell_d-1,\lambda_{d-1},\dotsc,\lambda_1) \leq g(d;\ell_d-1,\lambda_{d-1},\dotsc,\lambda_1),
\end{equation*}

\noindent
for all $\lambda_j \geq 0$, $j \in [1,d-1]$. As a result, we can determine a point $P_0$ of $V^\red$ by solving polynomials of degree at most $d$ when $r \geq g(V^\red)$. Taking $H$ to be a hyperplane which does not contain $P_0$, we can determine a point $P_1$ of $H \cap \mcc(V^\red;P_0)$ by solving polynomials of degree at most $d$ when $r \geq g\lp \mcc(V^\red;P_0) \rp + 1$. Therefore,
\begin{equation*}
	\Lambda(P_0,P_1) \subseteq \mcc(V^\red;P_0) \subseteq V^\red,
\end{equation*}

\noindent
and we can determine a point a point of $\Lambda(P_0,P_1) \cap V_d \subseteq V$ by solving an additional degree $d$ polynomial. Consequently,
\begin{equation*}
	r(d;\ell_d,\dotsc,\ell_1) \leq \max\lb g(V^\red), g\lp \mcc(V^\red;P_0) \rp + 1 \rb = g\lp \mcc(V^\red;P_0) \rp + 1 = g(V) = g(d;\ell_d,\dotsc,\ell_1).
\end{equation*}

Finally, we note that the polar cone construction introduces only finitely many hypersurfaces, all of which are strictly smaller degree. Consequently, iterating Lemma \ref{lem:The Reduction Lemma} yields that $g(d;\ell_d,\dotsc,\ell_1)$ is finite for every type $\lbm d &\cdots &1\\ \ell_d &\cdots &\ell_1 \rem$. 
\end{proof}

\subsection{Sylvester's Obliteration Algorithm}\label{subsec:Sylvester's Obliteration Algorithm}

In \cite{Sylvester1887}, Sylvester writes
\begin{quote}
	``In the following memoir I propose to present \emph{Hamilton's} process under what appears to me to be a clearer and more easily intelligible form, to extend his numerical results and to establish the principles of a more general method than that to which he has confined himself.''
\end{quote}

\noindent
We now propose to serve the analogous role for Sylvester that Sylvester served for Hamilton. Note that \cite{Sylvester1887} begins with a ``a somewhat more extended statement of the Law of Inertia (Tr\"{a}gheitsgesetz) for quadratic forms'' and provides a brief history of the theory of Tschirnhaus transformations, both of which we omit here. Sylvester's law of inertia is well-known (see \cite[Section 1]{Ostrowski1959}) and not necessary for our purposes. We refer the reader to \cite[Section 2 and  Appendix B]{Wolfson2021} for a more complete history of Tschirnhaus transformations.

Throughout this subsection, we consider a system $S = \lb f_1,\dotsc,f_s \rb$ of homogeneous polynomials. Given a solution $P_0$ of $S$, the ``first emanant'' \cite[p.471]{Sylvester1887} of $S$ at $P_0$ is
\begin{equation*}
	S(1;P_0) := \lb t(\ell,f_j,P_0) \st j \in [1,s], \ell \in [0,\deg(f_j)-1] \rb,
\end{equation*}
	
	\noindent
	where $t(\ell,f_j,P_0)$ is as in equation (\ref{eqn:jth polar of a polynomial}) of Definition \ref{def:Polars and Polar Cones}. Given a solution $P_1$ of $S(1;P_0)$, Sylvester's sub-lemma \cite[p.472]{Sylvester1887} states that any linear combination $\lambda_0 P_0 + \lambda_1 P_1$ (what he calls an ``alliance'' of $P_0$ and $P_1$) is a solution of $S(1;P_0)$, where $[\lambda_0:\lambda_1] \in \PP^1(K)$. Consequently, Sylvester says that $P_0$ and $P_1$ define a ``linear solution'' of $S(1;P_0)$ (and thus also of $S$, since $S \subseteq S(1;P_0)$).
	
	Note that the geometric version of Sylvester's sub-lemma \cite[p.472]{Sylvester1887} is Lemma \ref{lem:Bertini's Lemma for Intersections of Hypersurfaces}. The core algebraic computation reduces to the case of hypersurfaces; see Lemma 2.8 of \cite{Sutherland2021C}. Additionally, just as the second-named author constructs iterated polar cones in \cite{Sutherland2021C}, Sylvester analogously introduces ``$r^{th}$ emanants'' \cite[p.472]{Sylvester1887} and ``the Lemma'' \cite[p.472]{Sylvester1887} is the analogue of the polar point lemma (Lemma \ref{lem:Polar Point Lemma}). His proof follows from iterating the sublemma.
	
	Sylvester now focuses on linear solutions \cite[p.475]{Sylvester1887} of systems of equations. First, he introduces ``completed emanants'' \cite[p.475]{Sylvester1887} to ensure that $P_1$ is distinct from $P_0$ (and thus $P_0$ and $P_1$ determine a genuine linear solution). More specifically, a completed emanant is a system of equations $T = S(1;P_0) \cup \lb g \rb$, where $g$ is a homogeneous linear polynomials such that $g(P_0) \not= 0$. Next, let $S$ be of type $\lbm d &\cdots &1\\ \ell_d &\cdots &\ell_1 \rem$. Sylvester introduces notation \cite[p.475]{Sylvester1887} to denote the number of variables necessary to determine a linear solution of $S$. We modify his notation slightly for clarity and write $[d;\ell_d,\dotsc,\ell_1]$ instead of $[p,q,r,\dotsc,\eta,\theta]$. Note that
\begin{equation*}
	[d;\ell_d,\dotsc,\ell_1] = r\lp \mcc(\VV(S);P_0) \rp+1,
\end{equation*}

\noindent
for any $P_0 \in \VV(S)(K)$. It follows that Sylvester's formula of reduction \cite[p.475]{Sylvester1887} is
\begin{equation*}
	[d;\ell_d,\dotsc,\ell_1] \leq \left[d;\ell_d-1,\ell_d+\ell_{d-1},\dotsc,\SL_{j=2}^d \ell_j, \SL_{j=1}^d \ell_j \right] + 1,
\end{equation*}

\noindent
when $\ell_d \geq 2$. When $\ell_d=1$, let $d'$ be the largest $j \leq d-1$ such that $\ell_j$ is non-zero. Then, Sylvester's formula of reduction is
\begin{equation*}
[d;\ell_d,\dotsc,\ell_1] \leq \left[d';\ell_{d'},\ell_{d'}+\ell_{d'-1},\dotsc,\SL_{j=2}^{d'} \ell_j, \SL_{j=1}^{d'} \ell_j \right] + 1.
\end{equation*}

\noindent
Sylvester then claims the his formula of obliteration \cite[p.475]{Sylvester1887} without proof. We state his formula of obliteration and provide a proof, for the sake of completeness.

\begin{prop}\label{prop:Sylvester's Formula of Obliteration} \textbf{(Sylvester's Formula of Obliteration)}\\
Let $S$ be a system of homogeneous polynomials of type $\lbm d &\cdots &1\\ \ell_d &\cdots &\ell_1 \rem$ with $d \geq 2$ and $\ell_d \geq 2$. Then,
\begin{align*}
	[d;\ell_d,\dotsc,\ell_1] &\leq [d-1;\lambda_{d-1}, \lambda_{d-2},\dotsc,\lambda_2,\lambda_1] + \ell_d,\\
	&= [d-1;\lambda_{d-1}, \lambda_{d-2},\dotsc,\lambda_2,\lambda_1+\ell_d],
\end{align*}

\noindent
where
\begin{equation*}
	\lambda_{d-j} = \binom{\ell_d+j-1}{j} \frac{j\ell_d+1}{j+1} + \SL_{\nu=0}^{j-1} \binom{\ell_d+\nu-1}{\nu} \ell_{d-j+\nu}.
\end{equation*}
\end{prop}

\begin{proof}
It is straightforward to see that iteratively applying Sylvester's formula of reduction allows us to reduce to a system of equations of degree at most $d-1$. For the explicit numerics, we give a proof via induction on $\ell_d$. Note that to determine a linear solution of $S$, it suffices to determine a point solution of a completed emanant $T_0$ of $S$ at some point solution $P_0$. Additionally, we note that the type of $T_0$ is
\begin{equation*}
	\lbm d &d-1 &\cdots &2 &1\\ \ell_d &\ell_d+\ell_{d-1} &\cdots &\SL_{j=2}^d \ell_j &\lp \SL_{j=1}^d \ell_j \rp + 1 \rem.
\end{equation*}

Now, suppose that $\ell_d=1$. We can determine a point solution $P_1$ of $T_0$ by determining a linear solution of the subsystem $T_0'$, which is of type
\begin{equation*}
	\lbm d-1 &\cdots &2 &1\\ 1+\ell_{d-1} &\cdots &1+\SL_{j=2}^{d-1} \ell_j &\lp 1+\SL_{j=1}^{d-1} \ell_j \rp + 1 \rem.
\end{equation*}

\noindent
Futhermore, we see that
\begin{align*}
	\lambda_{d-j} = \binom{1+j-1}{j} \frac{j(1)+1}{j+1} + \SL_{\nu=0}^{j-1} \binom{1+\nu-1}{\nu} \ell_{d-j+\nu} = 1 + \SL_{\nu=0}^{j-1} \ell_{d-j+\nu} = 1 + \SL_{\mu=d-j}^{d-1} \ell_\mu,
\end{align*}

\noindent
so the claim holds when $\ell_d=1$. Now, consider the case where $\ell_d \geq 2$ is arbitrary. To determine a point solution of $T_0$, it suffices to determine a linear solution of a subsystem $T_0'$, which is of type
\begin{equation*}
	\lbm d &d-1 &\cdots &2 &1\\ \ell_d-1 &\ell_d+\ell_{d-1} &\cdots &\SL_{j=2}^d \ell_j &\lp \SL_{j=1}^d \ell_j \rp + 1 \rem .
\end{equation*}

\noindent
Thus,
\begin{align*}
	[d;\ell_d,\dotsc,\ell_1] \leq \left[d;\ell_d-1,(\ell_d+\ell_{d-1}),\dotsc,\lp \SL_{j=2}^d \ell_d \rp, \lp \SL_{j=1}^d \ell_j \rp + 1\right].
\end{align*}

\noindent
By induction, however, we have that
\begin{equation*}
	\left[d;\ell_d-1,(\ell_d+\ell_{d-1}),\dotsc,\lp \SL_{j=2}^d \ell_d \rp, \lp \SL_{j=1}^d \ell_j \rp + 1\right] \leq [d-1;\theta_{d-1},\dotsc,\theta_1+\ell_d],
\end{equation*}

\noindent
where
\begin{align*}
	\theta_{d-j} &= \binom{(\ell_d-1)+j-1}{j} \frac{j(\ell_d-1)+1}{j+1} + \SL_{\nu=0}^{j-1} \binom{(\ell_d-1)+\nu-1}{\nu} \lp \SL_{\mu=0}^j  \ell_{d-j+\mu}\rp, \\
	&= \binom{\ell_d+j-2}{j} \frac{j\ell_d-j+1}{j+1} + \SL_{\nu=0}^{j-1} \binom{\ell_d+\nu-2}{\nu} \lp \SL_{\mu=0}^j  \ell_{d-j+\mu}\rp.
\end{align*}

\noindent
Note that for each $\mu' \in [0,j-1]$, there are exactly $\mu'+1$ summands containing $\ell_{d-j+\mu'}$, namely
\begin{equation*}
	\binom{\ell_d-2}{0} \ell_{\mu'}, \binom{\ell_d-1}{1} \ell_{\mu'}, \dotsc, \binom{\ell_d + \mu'-2}{\mu'}\ell_{\mu'}.
\end{equation*}

\noindent
Additionally, there are exactly $j$ summands containing $\ell_d$, namely
\begin{equation*}
	\binom{\ell_d-2}{0} \ell_d, \binom{\ell_d-1}{1} \ell_d, \dotsc, \binom{\ell_d + j-3}{j-1}\ell_d.
\end{equation*}

\noindent
As a result,
\begin{align*}
	\theta_{d-j} &= \binom{\ell_d+j-2}{j} \frac{j\ell_d-j+1}{j+1} + \SL_{\nu'=0}^{j-1} \binom{\ell_d+\nu'-2}{\nu'}\ell_d +  \SL_{\mu_1=0}^{j-1} \lp \SL_{\mu_2=0}^{\mu_1} \binom{\ell_d+\mu_2-2}{\mu_2} \rp \ell_{d-j+\mu_1},\\
	&= \binom{\ell_d+j-2}{j} \frac{j\ell_d-j+1}{j+1} +  \binom{\ell_d+j-2}{j-1}\ell_d +  \SL_{\mu_1=0}^{j-1} \binom{\ell_d+\mu_1-1}{\mu_1} \ell_{d-j+\mu_1}.
\end{align*}

\noindent
Next, we see that
\begin{equation*}
	\binom{\ell_d+j-2}{j} \frac{j\ell_d-j+1}{j+1} = \binom{\ell_d+j-2}{j} \frac{j\ell_d+1}{j+1} - \binom{\ell_d+j-2}{j} \frac{j}{j+1},
\end{equation*}

\noindent
and
\begin{equation*}
	\binom{\ell_d+j-2}{j-1}\ell_d = \binom{\ell_d+j-2}{j-1} \frac{j\ell_d+1}{j+1} + \binom{\ell_d+j-2}{j-1} \frac{\ell_d-1}{j+1}.
\end{equation*}

\noindent
Noting that $\binom{\ell_d+j-2}{j} + \binom{\ell_d+j-2}{j-1}  = \binom{\ell_d+j-1}{j}$, it follows that
\begin{equation*}
	\theta_{d-j} = \binom{\ell_d+j-1}{j}\frac{j\ell_d+1}{j+1} + \binom{\ell_d+j-2}{j-1} \frac{\ell_d-1}{j+1} - \binom{\ell_d+j-2}{j} \frac{j}{j+1} + \SL_{\mu_1=0}^{j-1} \binom{\ell_d+\mu_1-1}{\mu_1} \ell_{d-j+\mu_1}.
\end{equation*}

\noindent
However,
\begin{align*}
\binom{\ell_d+j-2}{j-1} \frac{\ell_d-1}{j+1} - \binom{\ell_d+j-2}{j} \frac{j}{j+1} &= \frac{(\ell_d+j-2)!(\ell_d-1)}{(j-1)!(\ell_d-1)!(j+1)} - \frac{(\ell_d+j-2)!j}{j!(\ell_d-2)!(j+1)},\\
&= \frac{(\ell_d+j-2)!}{(j-1)!(\ell_d-2)!(j+1)} - \frac{(\ell_d+j-2)!}{(j-1)!(\ell_d-2)!(j+1)},\\
&= 0,
\end{align*}

\noindent
and thus
\begin{equation*}
	\theta_{d-j} = \binom{\ell_d+j-1}{j}\frac{j\ell_d+1}{j+1} + \SL_{\mu_1=0}^{j-1} \binom{\ell_d+\mu_1-1}{\mu_1} \ell_{d-j+\mu_1} = \lambda_{d-j},
\end{equation*}

\noindent
which proves the claim.
\end{proof}

Sylvester then applies his formula of obliteration to the question of determining non-zero solutions of equations which define the Tschirnhaus complete intersections $\tau_{1,\dotsc,m-1}$, including his Triangle of Obliteration. We omit his discussion here as the bounds he obtains are succeeded by the bounds of \cite{Brauer1975}, \cite{Wolfson2021}, \cite{Sutherland2021C}, and the next section.

\section{Upper Bounds on Resolvent Degree}\label{sec:Upper Bounds on Resolvent Degree}

\subsection{Previous Bounds}\label{subsec:Previous Bounds}

The current upper bounds on $\RD(n)$ were determined by the second-named author in \cite[Theorem 3.27]{Sutherland2021C}, which improved upon those of Wolfson \cite[Theorem 5.6]{Wolfson2021}. The general framework used by both the second-named author (with polar cones) and Wolfson (without polar cones) for constructing their respective bounding functions $G(m)$ and $F(m)$ was outlined in Remark \ref{rem:Strategy for Upper Bounds on RD(n)}. We define $G(m)$ below, but first we highlight the function's key properties (and recall that property 1, which both $F(m)$ and $G(m)$ share, is why we refer to $F(m)$ and $G(m)$ as bounding functions).

\begin{theorem}\label{thm:Sutherland2021C} \textbf{(Theorem 1.3 of \cite{Sutherland2021C})}\\
	The function $G(m)$ of \cite[Definition 3.26]{Sutherland2021C} has the following properties:
	\begin{enumerate}
		\item For each $m \geq 1$ and $n \geq G(m)$, $\RD(n) \leq n-m$.
		\item For each $d \geq 4$, $G(2d^2+7d+6) \leq \frac{(2d^2+7d+5)!}{d!}$. In particular, for $d \geq 4$ and $n \geq \frac{(2d^2+7d+5)!}{d!}$,
		\begin{equation*}
			\RD(n) \leq n-2d^2-7d-6.
		\end{equation*}
		\item For each $m \geq 1$, $G(m) \leq F(m)$ with equality only when $m \in \lb 1,2,3,4,5,15,16 \rb$ and
		\begin{align*}
			\lim\limits_{m \ra \infty} \frac{F(m)}{G(m)} = \infty.
		\end{align*}
	\end{enumerate}
\end{theorem}

We will now numerically define $G(m)$ (which will require two additional functions) and then a summary of the construction of $G(m)$. We refer the reader to \cite[Section 3]{Sutherland2021C} for the full construction of $G(m)$ and proofs of the statements in Theorem \ref{thm:Sutherland2021C}. 

\begin{definition}\label{def:The Function G(m)}  \textbf{(The Function $G(m)$)}\\
We first define $\vartheta:\ZZ_{\geq 3} \times \ZZ_{\geq 1} \ra \ZZ_{\geq 1}$ so that $\vartheta(d,k)$ is the minimal $r \in \ZZ_{\geq 1}$ such that
	\begin{equation*}
		(k+1)(r-k) - \SL_{j=2}^d \binom{k+i}{i} \geq 0.
	\end{equation*}
	
	\noindent
	Explicitly, we have
	\begin{align*}
		\vartheta(d,k) = k + \left\lceil \frac{1}{k+1}\lp \binom{k+d+1}{d} - (k+2) \rp \right\rceil.
	\end{align*}
	
	\noindent
	Next, we define $\varphi:\ZZ_{\geq 15} \times \ZZ_{\geq 1} \ra \ZZ_{\geq 1}$ by
	\begin{equation*}
		\varphi(d,k) = \max\lb \frac{(d+k)!}{d!}, \binom{\vartheta(d,k)+d+1}{d} - (\vartheta(d,k)+1)^2 - (\vartheta(d,k)+d) \rb.
	\end{equation*}
	
	\noindent
	Finally, we define $G:\ZZ_{\geq 1} \ra \ZZ_{\geq 1}$. For $m \in [1,14]$, we define $G(m)$ by
	\begin{center}
		\begin{tabular}{|c|ccccc|ccccc|}
		\hline
		$m$ &1 &2 &3 &4 &5 &6 &7 &8 &9 &10\\
		\hline
		$G(m)$ &2 &3 &4 &5 &9 &21 &109 &325 &1681 &15121\\
		\hline
		\end{tabular}
		\end{center}
		
		\begin{center}
		\begin{tabular}{|c|cccc|}
		\hline
		$m$ &11 &12 &13 &14 \\
		\hline
		$G(m)$ &151,201 &1,663,201 &19,958,401 &259,459,201 \\
		\hline
		\end{tabular}
	\end{center}
	
	\noindent
	and for $m \geq 15$ by
	\begin{align*}
		G(m) = 1 + \min\lb \varphi(d,m-d-1) \st 4 \leq d \leq m-1 \rb.
	\end{align*}
\end{definition}

The values of $G(m)$ for $m \in [1,5]$ are classical and described in \cite[Appendix B]{Wolfson2021}. In \cite{Chebotarev1954}, Chebotarev gave an argument that $\RD(n) \leq n-6$ for $n \geq 21$, however his argument had a gap which was fixed by \cite[Theorem 3.7]{Sutherland2021C}. More specifically, Chebotarev (like Wiman before him in \cite{Wiman1927}) assumed certain intersections of hypersurfaces were generic without proof.

For $m \in [6,14]$, the second-named author determined $k$-polar points on extended Tschirnhaus complete intersections $\tau_{1,\dotsc,d}^\circ$ \cite[Theorems 3.7, 3.10]{Sutherland2021C}. However, the degrees of iterated polar cones grow exponentially and this method could not be further extended \cite[Remark 3.19]{Sutherland2021C}. For general $m$, the second-named author was able to improve on the bounds of Wolfson by using \cite[Theorem 2.1]{DebarreManivel1998} to minimize the ambient dimension required for Wolfson's algorithm \cite[Theorem 3.24]{Sutherland2021C}.

\subsection{New Bounds}\label{subsec:New Bounds}

We will now improve on $G(m)$ for $m \in [13,17] \cup [22,25]$. For $m \in [7,16]$, $G(m)$ is obtained by determining an $(m-5)$-plane on $\tau_{1,2,3,4}^\circ$. Additionally, for $m \in [17,24]$, $G(m)$ is obtained by determining an $(m-6)$-plane on $\tau_{1,2,3,4,5}^\circ$. Finally, for $m \in [25,33]$, $G(m)$ is obtained by determining an $(m-7)$-plane on $\tau_{1,2,3,4,5,6}^\circ$.

Our improvements will come from determining an $(m-6)$-plane on $\tau_{1,2,3,4,5}^\circ$ for $m \in [13,17]$ and from determining an $(m-7)$-plane on $\tau_{1,2,3,4,5,6}^\circ$ for $m \in [22,25]$. Note that in each of these cases, one can apply the geometric obliteration algorithm to obtain improved bounds. However, we will use a slight modification which allows for a minor optimization.

\begin{remark}\label{rem:A Modification of the Geometric Obliteration Algorithm} \textbf{(A Modification of the Geometric Obliteration Algorithm)}\\
	Let $V \subseteq \PP_K^r$ be an intersection of hypersurfaces of type $\lbm d &\cdots &1\\ \ell_d &\cdots &\ell_1 \rem$. Recall that successive uses of Proposition \ref{prop:The Obliteration Proposition} yield that
\begin{equation*}
	g(V) = g\lp V_1^\Syl \rp = \cdots = g\lp V_{d-3}^\Syl \rp = g\lp V_{d-2}^\Syl \rp,
\end{equation*}	

\noindent
and that $V_{d-2}^\Syl$ is an intersection of type $\lbm 2 &1\\ \lambda_2 &\lambda_1 \rem$. In the spirit of the obliteration algorithm, we could indeed continue to apply Lemma \ref{lem:The Reduction Lemma} until there is a single quadric left, at which point we need only solve a final quadratic polynomial.

However, we also note that $\deg\lp V_{d-2}^\Syl \rp$ is $2^{\lambda_2}$ and thus we can determine a point of $W_V$ by solving a polynomial of degree $2^{\lambda_2}$ whenever $r \geq \lambda_2 + \lambda_1$. Consequently, we obtain a slight improvement in the forthcoming bounds on $\RD(n)$ by reducing only to a $j^{th}$ partial Sylvester reduction of $V_{d-2}^\Syl$ for some $j < \lambda_2$ instead of $V_{d-1}^\Syl$. 
\end{remark}

\begin{definition}\label{def:Optimal Reduction of Tschirnhaus Complete Intersection}  \textbf{(Optimal Reduction of Tschirnhaus Complete Intersection)}\\
	For each $d \geq 3$ and $m \geq d+2$, consider
	\begin{equation*}
		W = \lp \mcc^{m-d-1}(\tau_{1,\dotsc,d};P_0,\dotsc,P_{m-d-2}) \rp_{d-2}^\Syl,
	\end{equation*}	

	\noindent	
	a $(d-2)^{nd}$ Sylvester reduction of an $(m-d-1)^{st}$ polar cone of $\tau_{1,\dotsc,d}$, which is of type $\lbm 2 &1\\ \lambda_2 &\lambda_1 \rem$. For each $j \in [1,\lambda_2-1]$, note that a $j^{th}$ partial Sylvester reduction $W^\Syl(2;j)$ of $W$ has type $\lbm 2 &1\\ \lambda_2 - j &\lambda_1 + \SL_{\nu=\lambda_2-j}^{\lambda_2-1} \nu \rem$. Further,
\begin{equation*}
\deg\lp W^\Syl(2;j) \rp = 2^{\lambda_2-j}.
\end{equation*}	

\noindent
For each such $j$, set
	\begin{equation*}
		\xi(m,d;j) := \max\lb (m-d+1)+\lp \lambda_2-j \rp + \lp \lambda_1 + \SL_{\nu=\lambda_2-j}^{\lambda_2-1} \nu \rp, 2^{\lambda_2-j}+1 \rb. 
	\end{equation*}
	
	\noindent
	The \textbf{optimal reduction bound of $\tau_{1,\dotsc,d}$ for $m$}, is
	\begin{align*}
		\Xi(m,d) := \min\lb \xi(m,d;j) \st j \in [0,\lambda_2-1] \rb.
	\end{align*}
\end{definition}

In particular, $\Xi(m,d)$ is defined exactly so that for $n \geq \Xi(m,d)$, we can determine an $(m-d-1)^{th}$ polar point of $\tau_{1,\dotsc,d}^\circ$ in $\PP_{K_n}^{n-1}$ over an extension $K'/K_n$ with $\RD(K'/K_n) \leq \RD(\Xi(m;d))$.

\begin{remark}\label{rem:Xi(m,d) is Non-Decreasing in m} \textbf{($\Xi(m,d)$ is Non-Decreasing in $m$)}\\
	Note that $\Xi(m,d)$ is non-decreasing in $m$ for fixed $d$. This can be seen geometrically from the fact if $(P_0,\dotsc,P_{m-d-1})$ is an $(m-d-1)^{st}$ polar point of $\tau_{1,\dotsc,d}$, then $(P_0,\dotsc,P_{m-d-2})$ must be an $(m-d-2)^{nd}$ polar point of $d$, so $\Xi(m,d) \geq \Xi(m-1,d)$.
\end{remark}

 We are now ready to state and prove the main theorem.

\begin{theorem}\label{thm:Bounds from the Geometric Obliteration Algorithm}  \textbf{(Bounds from the Geometric Obliteration Algorithm)}
\begin{enumerate}
	\item For $n \geq 5,250,198$, $\RD(n) \leq n-13$.
	\item For each $m \in [14,17]$ and $n > \frac{(m-1)!}{120}$, $\RD(n) \leq n-m$.
	\item For $n \geq 381,918,437,071,508,900$, $\RD(n) \leq n-22$. 
	\item For each $m \in [23,25]$ and $n > \frac{(m-1)!}{720}$, $\RD(n) \leq n-m$. 
\end{enumerate}
\end{theorem}

\begin{proof}
We continue to use the notation established in Definition \ref{def:Optimal Reduction of Tschirnhaus Complete Intersection}. For each $m \in [13,17]$, we set
\begin{equation*}
	G'(m) = \max\lb \Xi(m,5), \frac{(m-1)!}{120}+1 \rb,
\end{equation*}

\noindent
and for each $m \in [22,25]$, we set
\begin{equation*}
	G'(m) = \max\lb \Xi(m,6), \frac{(m-1)!}{720}+1 \rb. 
\end{equation*}

\noindent
In each case, it suffices to show the claim when $n=G'(m)$. Further, note that $G'(m) = \Xi(m,5)$ exactly when $m=13$ and $G'(m)=\Xi(m,6)$ exactly when $m=22$; this claim is justified by explicit computation and is given in the tables at the end of the proof. Recall that the space of Tschirnhaus transformations up to re-scaling is $\PP_{K_{G'(m)}}^{G'(m)-1}$.

Let us first consider the case of $m \in [13,17]$ and let $H \subseteq \PP_{K_n}^{G'(m)-1}$ be a hyperplane which does not contain $[1:0:\cdots:0]$. Note that $H \cong \PP_{K_n}^{G'(m)-2}$ and $H \cap \tau_{1,\dotsc,5} = H \cap \tau_{1,\dotsc,5}^\circ$. Since $\Xi(m,5) \geq \Xi(m-1,5)$, we can assume that we have an $(m-7)$-polar point $(P_0,\dotsc,P_{m-7})$ of $H \cap \tau_{1,\dotsc,5}^\circ$. Consider the minimal $j$ such that $\Xi(m,5) = \xi(m,5;j)$. By definition of $\xi(m,5;j)$, we have that
\begin{equation*}
	\dim\lp \lp \lp \mcc^{m-6}(H \cap \tau_{1,\dotsc,5}^\circ;P_0,\dotsc,P_{m-7}) \rp_3^\Syl \rp^\Syl(2;j) \rp \geq m-6.
\end{equation*}

\noindent
Since $\dim\lp \Lambda(P_0,\dotsc,P_{m-7}) \rp = m-7$, we can determine a point of
\begin{equation*}
	\mcc^{m-6}\lp H \cap \tau_{1,\dotsc,5}^\circ;P_0,\dotsc,P_{m-7} \rp  \setminus \Lambda\lp P_0,\dotsc,P_{m-7} \rp,
\end{equation*}

\noindent
by solving a polynomial of degree at most $\Xi(m;5)$. By construction $\lp P_0,\dotsc,P_{m-6} \rp$ is an $(m-6)$-polar point and Lemma \ref{lem:Polar Point Lemma} yields that $\Lambda = \Lambda(P_0,\dotsc,P_{m-6}) \subseteq \tau_{1,\dotsc,5}^\circ$ is an $(m-6)$-plane. We can then determine a point of $\Lambda \cap \tau_{1,\dotsc,m-1}^\circ$ by solving a polynomial of degree $\frac{(m-1)!}{120}$.

We now consider the similar case of $m \in [22,25]$. Let $H \subseteq \PP_{K_n}^{G'(m)-1}$ be a hyperplane which does not contain $[1:0:\cdots:0]$. Note that $H \cong \PP_{K_n}^{G'(m)-2}$ and $H \cap \tau_{1,\dotsc,5} = H \cap \tau_{1,\dotsc,5}^\circ$. Since $\Xi(m,6) \geq \Xi(m-1,6)$, we can assume that we have an $(m-8)$ polar point $(P_0,\dotsc,P_{m-8})$ of $H \cap \tau_{1,\dotsc,6}^\circ$. Consider the minimal $j$ such that $\Xi(m,6) = \xi(m,6;j)$. Observe that
\begin{equation*}
	\dim\lp \lp \lp \mcc^{m-7}(H \cap \tau_{1,\dotsc,6}^\circ;P_0,\dotsc,P_{m-8}) \rp_4^\Syl \rp^\Syl(2;j) \rp \geq m-7,
\end{equation*}

\noindent
and so we can determine a point $P_{m-6}$ of
\begin{equation*}
	\mcc^{m-7}\lp H \cap \tau_{1,\dotsc,6}^\circ;P_0,\dotsc,P_{m-8} \rp \setminus \Lambda\lp P_0,\dotsc,P_{m-8} \rp,
\end{equation*}

\noindent
by solving a polynomial of degree at most $\Xi(m;6)$. it follows that $\lp P_0,\dotsc,P_{m-7}\rp$ is an $(m-7)$-polar point of $\tau_{1,\dotsc,6}^\circ$ and so $\Lambda = \Lambda(P_0,\dotsc,P_{m-7}) \subseteq \tau_{1,\dotsc,6}^\circ$ is an $(m-7)$-plane. Consequently, we can determine a point of $\Lambda \cap \tau_{1,\dotsc,m-1}^\circ$ by solving a polynomial of degree $\frac{(m-1)!}{720}$.

We now show that $G'(m) = \Xi(m,5)$ exactly when $m=13$ and $G'(m)=\Xi(m,6)$ exactly when $m=22$. In the following tables, we note the values of $\Xi(m,5)$ and $\frac{(m-1)!}{120}+1$ for $m \in [13,17]$ and the approximate values of $\Xi(m,6)$ and $\frac{(m-1)!}{720}+1$ for $m \in [22,25]$. The exact values of $\Xi(m,5)$ for $m \in [13,17]$ and of $\Xi(m,6)$ for $m \in [22,25]$ were computed using Algorithm \ref{alg:The Geometric Obliteration Algorithm Applied to Tschirnhaus Complete Intersections}, which can be found in Subsection \ref{subsec:Appendix D - The Geometric Obliteration Algorithm Applied to Polar Cones of Tschirnhaus Complete Intersections}.

\begin{center}

\begin{tabular}{cc}

\begin{tabular}{|c|c|c|}
\hline
$m$ &$\Xi(m,5)$ &$\frac{(m-1)!}{120}+1$ \\
\hline
13 &5,250,198 &3,991,681 \\
\hline
14 &12,253,482 &51,891,841 \\
\hline
15 &26,357,165 &726,485,761 \\
\hline
16 &53,008,668 &10,897,286,401 \\
\hline
17 &100,769,994 &174,356,582,401 \\
\hline
\end{tabular}

&
\begin{tabular}{|c|c|c|}
\hline
$m$ &$\Xi(m,6)$ &$\frac{(m-1)!}{720}+1$\\
\hline
22 &$\sim 3.819 \times 10^{17}$ &$\sim 7.096 \times 10^{16}$ \\
\hline
23 &$\sim 9.526 \times 10^{17}$ &$\sim 1.561 \times 10^{18}$ \\
\hline
24 &$\sim 2.262 \times 10^{18}$ &$\sim 3.591 \times 10^{19}$ \\
\hline
25 &$\sim 5.137 \times 10^{18}$ &$\sim 8.617 \times 10^{20}$ \\
\hline
\end{tabular}

\end{tabular}

\end{center}
\end{proof}

\subsection{Obstruction to Further Bounds via the Geometric Obliteration Algorithm}\label{subsec:Obstruction to Further Bounds via The Geometric Obliteration Algorithm}

Unfortunately, the proof strategy of Theorem \ref{thm:Bounds from the Geometric Obliteration Algorithm} does not yield further bounds on $\RD(n)$. Recall that for $m \geq 15$, $G(m)$ is defined by
\begin{equation*}
	G(m) = 1 + \min\lb \varphi(d,m-d-1) \st d \in [4,m-1] \rb,
\end{equation*}

\noindent
where
\begin{equation*}
	\varphi(d,k) = \max\lb \frac{(d+k)!}{d!}, \binom{\vartheta(d,k)+d+1}{d} - ( \vartheta(d,k)+1 )^2 - (\vartheta(d,k)+d) \rb.
\end{equation*}

\noindent
For each $d$, the values of $m$ for which $G(m) = 1 + \varphi(d,m-d-1)$ is a set of consecutive integers. Equivalently, there are positive integers $m_d$ and $m_d'$ such that $G(m) = 1 + \varphi(d,m-d-1)$ if and only if $m \in \left[ m_d,m_d' \right]$; see \cite[Lemma 3.33]{Sutherland2021C} for details. 

Similarly, we briefly introduce the notation
\begin{equation*}
	\varrho(d,k) = \max\lb \Xi(d+k+1,d), \frac{(d+k)!}{d!}+1 \rb
\end{equation*}

\noindent
for $d \geq 4$ and $k \geq 1$, as well as
\begin{equation*}
	H(m) = \min\lb \varrho(d,m-d-1) \st d \in [4,m-1] \rb
\end{equation*}

\noindent
for $m \geq 13$. For fixed $d$, note that $\Xi(m,d)$ is a polynomial in $m$, whereas $\frac{(d+k)!}{d!} = \frac{(m-1)!}{d!}$ grows factorially. It follows that for each $d$, there are positive integers $M_d$ and $M_d'$ such that $H(m) = \varrho(d,m-d-1)$ if and only if $m \in [M_d,M_d']$. 

In the following table, we compare the values $m_d$ and $M_d$ for $d=5,6,7,8$. 

\begin{center}
\begin{tabular}{|c|c|c|}
\hline
$d$ &$m_d$ &$M_d$\\
\hline
$5$ &17 &13 \\
\hline 
$6$ &25 &22 \\
\hline
$7$ &34 &41 \\
\hline
$8$ &44 &78 \\
\hline
\end{tabular}
\end{center}

This provides further evidence, along with \cite[Remark 3.19]{Sutherland2021C}, that iterated polar cone methods are most effective for intersections of hypersurfaces of small types. Next, we determine an explicit lower bound on $\Xi(m,d)$. 

\begin{lemma}\label{lem:Lower Approximation} \textbf{(Lower Approximation)}\\
	Let $V \subseteq \PP_K^r$ be an intersection of hypersurfaces of type $\lbm d\\ \ell_d \rem$ with $d \geq 3$ and $\ell_d \geq 2$. Denote the type of a $(d-2)^{nd}$ Sylvester reduction $V_{d-2}^\Syl$ by $\lbm 2 &1\\ \lambda_2 &\lambda_1 \rem$. Then, 
	\begin{align*}
		\lambda_1 \geq \lambda_2 \geq \left\lceil 2^{5-2d} \lp \ell_d-1 \rp^{2d-4} \right\rceil.
	\end{align*}
\end{lemma}

\begin{proof}
Note that the number of degree $d-1$ hypersurfaces of $V_1^\Syl$ is
\begin{equation*}
	\theta_{d-1} = \SL_{j=1}^{\ell_d-1} \ell_d-j = \frac{1}{2}(\ell_d-1)\ell_d \geq \left\lceil \frac{1}{2}(\ell_d-1)^2 \right\rceil.
\end{equation*}

\noindent
The same argument yields that the number of degree $d-2$ hypersurfaces of $V_2^\Syl$ is
\begin{equation*}
	\theta_{d-2} \geq \left\lceil \frac{1}{2}\left\lceil \frac{1}{2}(\ell_d-1)^2 \right\rceil^2 \right\rceil \geq \left\lceil 2^{-3} (\ell_d-1)^4 \right\rceil.
\end{equation*}

\noindent
Proceeding similarly, we see that
\begin{equation*}
	\lambda_2 = \theta_2 \geq \left\lceil 2^{5-2d} \lp \ell_d-1 \rp^{2d-4} \right\rceil.
\end{equation*}

\noindent
Finally, note that $\lambda_1 \geq \lambda_2$ follows immediately from the polar cone construction.
\end{proof}

\begin{corollary}\label{cor:Lower Bound for Xi(m,d)}  \textbf{(Lower Bound for $\Xi(m,d)$)}\\
	Let $d \geq 4$ and $m \geq d+2$. Then,
	\begin{equation*}
		\Xi(m,d) \geq \left\lceil 4 \lp \frac{m-d-1}{2} \rp^{2d-4} \right\rceil.
	\end{equation*}
\end{corollary}

\begin{proof}
	First, Proposition 2.26 of \cite{Sutherland2021C} yields that an $(m-d-1)^{th}$ polar cone of $\tau_{1,\dotsc,d}$ is of type
	\begin{equation*}
		\lbm d &d-1 &\cdots &2 &1\\ 1 &\binom{m-d}{1} &\cdots &\binom{m-3}{d-2} &\binom{m-2}{d-1} \rem.
	\end{equation*}
	
	\noindent
	Thus, the number of degree $d-1$ hypersurfaces of $V = \lp \tau_{1,\dotsc,d} \rp_1^\Syl$ is $m-d$. Let $\sigma(m,d)$ be as in Definition \ref{def:Optimal Reduction of Tschirnhaus Complete Intersection}. It follows from Lemma \ref{lem:Lower Approximation} that
	\begin{equation*}
		\lambda_1 \geq \lambda_2 \geq \left\lceil 2^{5-2d}(m-d-1)^{2d-4} \right\rceil.
	\end{equation*}
	
	\noindent
	Moreover, for each $j$,
	\begin{equation*}
		\xi(m,d;j) \geq \lambda_1 + \lambda_2 \geq \left\lceil 2^{5-2d}(m-d-1)^{2d-4} \right\rceil + \left\lceil 2^{5-2d}(m-d-1)^{2d-4} \right\rceil \geq \left\lceil 4 \lp \frac{m-d-1}{2} \rp^{2d-4} \right\rceil,
	\end{equation*}
	
	\noindent
	and thus it follows that
	\begin{equation*}
		\Xi(m,d) = \min\lb \xi(m,d;j) \st 0 \leq j \leq \lambda_2-1 \rb \geq \left\lceil 4 \lp \frac{m-d-1}{2} \rp^{2d-4} \right\rceil.
		\qedhere
	\end{equation*}
\end{proof}

While we do not provide a full comparison here, we note that the key obstruction to obtaining further bounds on $\RD(n)$ using the methods of Theorem \ref{thm:Bounds from the Geometric Obliteration Algorithm} is that $\Xi(m,d)$ has a lower bound which grows exponentially in $d$ and that $m-d-1$ grows much more quickly than $d$ (for example, $m-d-1 \geq 19$ for $m \geq 26$).

Having indicated the obstruction to obtaining further upper bounds on $\RD(n)$ using these methods, we now combine Theorems \ref{thm:Sutherland2021C} and \ref{thm:Bounds from the Geometric Obliteration Algorithm} to immediately construct a new bounding function with the same key properties of $G(m)$.

\begin{corollary}\label{cor:The New Bounding Function} \textbf{(The New Bounding Function)}\\
	Let $G':\ZZ_{\geq 2} \ra \ZZ_{\geq 1}$ be the function with
	\begin{equation*}
	G'(m) = \max\lb \Xi(m,5), \frac{(m-1)!}{120}+1 \rb,
	\end{equation*}

\noindent
for $m \in [13,17]$, with
\begin{equation*}
	G'(m) = \max\lb \Xi(m,6), \frac{(m-1)!}{720}+1 \rb,
\end{equation*}

\noindent
for $m \in [22,25]$, and with $G'(m) = G(m)$ for $m \not\in [13,17] \cup [22,25]$. Then, $G'(m)$ has the following properties:
\begin{enumerate}
	\item For each $m \geq 1$ and $n \geq G'(m)$, $\RD(n) \leq n-m$.
	\item For each $d \geq 4$, $G'(2d^2+7d+6) \leq \frac{(2d^2+7d+5)!}{d!}$. In particular, for $d \geq 4$ and $n \geq \frac{(2d^2+7d+5)!}{d!}$,
	\begin{align*}
		\RD(n) \leq n - 2d^2-7d-6.
	\end{align*}
\end{enumerate}
\end{corollary}

\subsection{Remaining Questions}\label{subsec:Remaining Questions}

To the best of the authors' knowledge, the bounding function $G'(m)$ of Corollary \ref{cor:The New Bounding Function} exhausts the techniques and methods for determining upper bounds on resolvent degree from the classical literature (including \cite{Bring1786,Chebotarev1954,Hamilton1836,Hilbert1927,Segre1945,Sylvester1887,
SylvesterHammond1887,SylvesterHammond1888,Tschirnhaus1683,Wiman1927}), as well as the modern insights from \cite{Brauer1975,Sutherland2021C,Wolfson2021}.

The bounding functions of Brauer, Hamilton, Sylvester, Wolfson, and the second-named author are constructed by determining points on the Tschirnhaus complete intersections $\tau_{1,\dotsc,m-1}^\circ$ over extensions of bounded resolvent degree. However, there are solutions of the quintic and the sextic which use alternative constructions of Tschirnhaus transformations (see \cite{Klein1884,Klein1905} for the respective original works or \cite{Morrice1956,Sutherland2019} for the respective English translations). We believe it would be insightful to understand whether one can reduce the general question of determining $\RD(n)$ to the more specific question of determining points on the Tschirnhaus complete intersections $\tau_{1,\dotsc,m-1}^\circ$.

\begin{question}\label{question:Optimal Formulas via Tschirnhaus Complete Intersections} \textbf{(Optimal Formulas via Tschirnhaus Complete Intersections)}\\
	For every $n$, let $m_n$ be such that $\RD(n) \leq n-m_n$. Is there a formula in $n-m_n$ variables for the general degree $n$ polynomial obtained by determining a point of $\tau_{1,\dotsc,m_n-1}^\circ$ over an extension $K'/K_n$ of bounded resolvent degree?
\end{question}

For general $m$, the definition of $G'(m)=G(m)$ uses the combinatorial condition of \cite[Theorem 2.1]{DebarreManivel1998} to guarantee the existence of $k$-planes on the $\tau_{1,\dotsc,d}^\circ$ and then uses the dimension of the relevant moduli space \cite[Subsection 3.3]{Sutherland2021C}. Notably, this combinatorial condition is non-constructive and relies only on the type of $\tau_{1,\dotsc,d}^\circ$. One might hope that such formulas could be determined using constructive methods and one approach may be to leverage the specific geometry of the $\tau_{1,\dotsc,d}^\circ$ (e.g., using more information than its type). 

\begin{question}\label{question:RD Bounds via Explicit Constructions of k-Planes}  \textbf{(RD Bounds via Explicit Constructions of $k$-Planes)}\\
	Is there a bounding function $\mfg(m)$ with $\mfg(m) \leq G'(m)$ which arises from an explicit construction of $k$-planes on the $\tau_{1,\dotsc,d}^\circ$? If so, is it possible to determine the bounding function $\mfg(m)$ such that
\begin{equation*}
	\lim\limits_{m \ra \infty} \frac{G(m)}{\mfg(m)} = \lim\limits_{m \ra \infty} \frac{G'(m)}{\mfg(m)} = \infty?
\end{equation*}	
\end{question}

Theorem \ref{thm:Bounds from the Geometric Obliteration Algorithm} was proved using a consequence of the geometric obliteration algorithm, namely that $r(V) \leq g(V)$ for any intersection of hypersurfaces $V$. Further examination of the relationship between $r(V)$ and $g(V)$ is of interest.

\begin{question}\label{question:Minimal Dimension Bound vs. Geometric Dimension Bound} \textbf{(Minimal Dimension Bound vs. Geometric Dimension Bound)}\\
	For which intersections of hypersurfaces $V$ is the inequality $r(V) \leq g(V)$ strict? Are there classical examples of types of intersections of hypersurfaces where the inequality is not strict?
\end{question}

Let us now briefly consider a cubic hypersurface $H = \VV(f) \subseteq \PP_K^r$. When $r=3$ and $H$ is smooth, the Cayley-Salmon theorem yields that $H$ contains exactly 27 lines. The resolvent degree of determining a line on $H$ is at most 3, as was established by Farb and Wolfson \cite[Theorem 8.2]{FarbWolfson2019}. Additionally, that $H$ has exactly 27 lines is consistent with \cite[Theorem 2.1]{DebarreManivel1998}, which states that the Fano variety of lines of a cubic surface in $\PP_K^3$ is non-empty and has dimension 0. In particular, when $r=3$, most points $P \in H(K)$ do not lie on a line of $H$ over an algebraic closure $\overline{K}$. When $r=4$, however, any polar cone $\mcc(V;P)$ has dimension at least one and thus every point $P \in V(H)$ lies on at least one line $\Lambda  = \Lambda(P,Q) \subseteq H$ over an algebraic closure $\overline{K}$. To determine such a point $Q$ directly, we must solve a polynomial of degree $6 = 3! = \deg(\mcc(V;P))$. Hence, we can determine a line through any point $P$ over an extension with $K'/K$ with $\RD(K'/K) \leq \RD(6) \leq 2$. 

Additionally, observe that
\begin{align*}
	g( \mcc(V;P) ) = g(3;1,1,1) = g(2;1,3) = 5.
\end{align*}

\noindent
Thus, when $r \geq 5$, we can determine a point $Q \in \mcc(V;P) \setminus \{P\}$ over an extension determined by solving at most cubic polynomials (i.e., over a solvable extension).

Now, let $V \subseteq \PP_K^r$ be an intersection of hypersurfaces of type $\lbm d &\cdots &1\\ \ell_d &\cdots &\ell_1 \rem$. For each $k \geq 1$, take $s_k(V)$ to be the minimal $s$ such that
\begin{align*}
	(k+1)(s-k) - \SL_{j=1}^d \ell_j \binom{k+j}{j} \geq 0. 
\end{align*}

One implication of Theorem 2.1 of \cite{DebarreManivel1998} is that $V$ contains a $k$-plane for all $r \geq s_k(V)$. We expect $s_k(V)$ to be the minimal ambient dimension required for $V$ to contain a $k$-plane; however, we expect the resolvent degree of determining such a $k$-plane to be large. Conversely, we expect $r\lp \mcc^k(V;P_0,\dotsc,P_{k-1}) \rp+k$, the ambient dimension required to determine a $k$-polar point over an extension $K'/K$ of small resolvent degree $(\RD(K'/K) \leq \RD(d))$, to be large. 

\begin{question}\label{question:Minimizing Ambient Dimension vs. Minimizing RD of Extensions} \textbf{(Minimizing Ambient Dimension vs. Minimizing RD of Extensions)}\\
Let $V$ be an intersection of hypersurfaces. How do $g\lp \mcc^k(V;P_0,\dotsc,P_{k-1}) \rp+k$, $r\lp \mcc^k(V;P_0,\dotsc,P_{k-1}) \rp+k$, and $s_k(V)$ compare?
\end{question}

Finally, we recall that we have worked entirely in characteristic zero (more specifically, over $\CC$). As we discussed in \ref{subsec:Resolvent Degree}, we do not lose any generality from the perspective of resolvent degree, as
\begin{equation*}
	\RD(n) = \RD_\CC\lp S_n \rp \geq \RD_K\lp S_n \rp
\end{equation*}

\noindent
by \cite[Theorem 1.3]{Reichstein2022}, with equality when $K$ has characteristic zero by \cite[Theorem 1.2]{Reichstein2022}. The foundational result for the polar cone framework we use is the technical lemma \cite[Lemma 2.8]{Sutherland2021C}. For those who wish to work in characteristic $p$, one would need to be careful of how the relevant combinatorics, such as \cite[Proposition 2.26]{Sutherland2021C}, change. Additionally, the modern reference for Tschirnhaus transformations \cite{Wolfson2021} works over $\ZZ$. To consider Tschirnhaus transformations in characteristic $p$, one would need to give extra consideration to the Tschirnhaus hypersurfaces of degree $p^k$.

\newpage

\section{Python Implementations of the Obliteration Algorithm and Related Phenomena}

In Subsection \ref{subsec:Appendix A - The Geometric Obliteration Algorithm}, we provide an implementation (Algorithm \ref{alg:The Geometric Obliteration Algorithm}) of the geometric obliteration algorithm in Python. In Subsection \ref{subsec:Appendix B - Lemmata for Appendix C}, we prove several lemmata which make the computations for the proof of Theorem \ref{thm:Bounds from the Geometric Obliteration Algorithm} feasible. Algorithm \ref{alg:The Geometric Obliteration Algorithm with Computational Improvements} in Subsection \ref{subsec:Appendix C - The Geometric Obliteration Algorithm with Computational Improvements} takes the same input and provides the same output as Algorithm \ref{alg:The Geometric Obliteration Algorithm}, but uses the lemmata of Subsection \ref{subsec:Appendix B - Lemmata for Appendix C} to decrease computation time. Finally, Algorithm \ref{alg:The Geometric Obliteration Algorithm Applied to Tschirnhaus Complete Intersections} in Subsection \ref{subsec:Appendix D - The Geometric Obliteration Algorithm Applied to Polar Cones of Tschirnhaus Complete Intersections} computes the information necessary for Theorem \ref{thm:Bounds from the Geometric Obliteration Algorithm}.

\subsection{Appendix A: The Geometric Obliteration Algorithm}\label{subsec:Appendix A - The Geometric Obliteration Algorithm}
\begin{algorithm}

\begin{algorithmname}\label{alg:The Geometric Obliteration Algorithm} \textbf{(The Geometric Obliteration Algorithm)}
\end{algorithmname}

\vspace{-12pt}

\hrulefill

\vspace{-12pt}

\begin{itemize}
	\item Input: An intersection of hypersurfaces $V$ of type $\lbm d &d-1 &\cdots &2 &1\\ \ell_d &\ell_{d-1} &\cdots &\ell_2 &\ell_1 \rem$ with $d \geq 2$, encoded as the list $\text{DegreeList} = [\ell_d, \ell_{d-1}, \dotsc, \ell_2, \ell_1]$.
	\item Output: The geometric dimension bound $g(d;\ell_d,\dotsc,\ell_1)$.  
\end{itemize}

\vspace{-12pt}

\hrulefill

\begin{algorithmic}[1]
	\Statex The function \textproc{ComputePolarCone} inputs a list which contains the type of an intersection of hypersurfaces $W$. It then returns a list which contains the type of a polar cone $\mcc(W;P)$. In particular, recall that for each $d' < d$, each hypersurface $H$ with $\deg(H) > d'$ defining $W$ contributes exactly one new degree $d'$ hypersurface defining $\mcc(W;P)$ and each hypersurface defining $\mcc(W;P)$ arises in this manner.
	
	\Statex
	
	\Function{ComputePolarCone}{List}:
		\State counter = List[0]
		\State ReturnList = [counter]
		\For{index \In \ \Range(1,\Len(List)):}
			\State counter += List[index]
			\State ReturnList.append(counter)
		\EndFor
		\State \Return ReturnList
	\EndFunction

	\Statex
	
	\Statex The function \textproc{ObliterateLargestDegreeHypersurfaces} inputs a list which contains the type of an intersection of hypersurfaces $W$ whose largest degree hypersurface has degree $d \geq 3$. It identifies the number of hypersurfaces of largest degree and proceeds to iteratively remove a hypersurface $H$ of largest degree and compute a polar cone of the remaining intersection of hypersurfaces $W'$ (with an additional hyperplane included).
	
	\Statex
	
	\Statex Note that an additional hyperplane is added each time to avoid repeated polar cone points, i.e. if $P$ was the cone point of the previous polar cone point, we pass to a hyperplane which does not contain $P$ to ensure that the cone point $Q$ of the next polar cone satisfies $Q \not= P$. Also, the polar cone of a hyperplane plane at any point is just the hyperplane itself, so to compute the combinatorics, it suffices to add one after computing the polar cone instead of doing it beforehand.
	
	\Statex
	
	\Statex As taking the polar cone of a hypersurface $H$ introduces only hypersurfaces of strictly smaller degree, this process terminates and \textproc{ObliterateLargestDegreeHypersurfaces} returns a list whose data is the multi-degree of an intersection of hypersurfaces $V'$ whose largest degree hypersurface has degree $d-1$. 
	
\algstore{bkbreak}
\end{algorithmic}
\end{algorithm}
	
\begin{algorithm}
\begin{algorithmic}[1]
\algrestore{bkbreak}
		
	\Function{ObliterateLargestDegreeHypersurfaces}{List}:
		\While{List[0] $>$ 0:}
			\State List[0] -= 1
			\State TempList = \textproc{ComputePolarCone}(List)
			\State List = TempList
			\State List[\Len(List)-1] += 1
		\EndWhile
		\State ReturnList = []
		\For{index \In \ \Range(1,\Len(List):}
			\State ReturnList.append(List[index])
		\EndFor
		\State \Return ReturnList
	\EndFunction
	
	\Statex

	\Statex The function \textproc{ObliterateQuadricsViaLoops} works similarly to \textproc{ObliterateLargestDegreeHypersurfaces}, but the input is the multi-degree of an intersection of hypersurfaces of type $\lbm 2 &1\\ \ell_2 &\ell_1 \rem$ and the loop ends with a single quadric remaining instead of zero quadrics remaining. 
	
	\Statex
	
		\Function{ObliterateQuadricsViaLoops}{List}:
		\While{List[0] $>$ 1:}
			\State List[0] -= 1
			\State TempList = \textproc{ComputePolarCone}(List)
			\State List = TempList
			\State List[\Len(List)-1] += 1
		\EndWhile
		\State \Return [List[0],List[1]]
	\EndFunction
	
	\Statex
	
	\Statex The procedure \textproc{Main} inputs the multi-degree of an intersection of hypersurfaces $V$ as the list DegreeList and proceeds to successively ``obliterate'' the hypersurfaces of largest degree. The final step of the procedure is to return a list of the form $[1,\alpha]$, which is the requisite intersection of a single quadric and $\alpha$ hyperplanes.
	
	\Statex
	
	\Procedure{Main}{DegreeList}:
		\For{index \In \ \Range(1,\Len(DegreeList)-1):}
			\State TempDegreeList = \textproc{ObliterateLargestDegreeHypersurfaces}(DegreeList)
			\State DegreeList = TempDegreeList
		\EndFor
		\State FinalList = \textproc{ObliterateQuadricsViaLoops}(DegreeList)
		\State Sum = FinalList[0] + FinalList[1]
		\State \Return Sum
	\EndProcedure
\end{algorithmic}

\end{algorithm}

\newpage

\subsection{Appendix B: Lemmata for Computational Improvements}\label{subsec:Appendix B - Lemmata for Appendix C}

In this subsection, we give explicit numerics for Proposition \ref{prop:The Obliteration Proposition} when $d = 2,3,4$.

\begin{lemma}\label{lem:Obliterating Quadrics} \textbf{(Obliterating Quadrics)}\\
	Consider an intersection of hypersurfaces $V$ of type $\lbm 2 &1\\ \ell_2 &\ell_1 \rem$. Then,
	\begin{equation*}
		 g(V) = 1 + \ell_1 + \frac{1}{2}(\ell_2-1)(\ell_2+2).
	\end{equation*}
\end{lemma}

\begin{proof}
First, observe that $V^\Syl(2;1)$ has type
\begin{equation*}
	\lbm 2 &1\\ \ell_2-1 &\ell_1 + \ell_2 \rem,
\end{equation*}
	
\noindent
by Definition \ref{def:Sylvester Reductions}. Similarly, $V^\Syl(2;2)$ has type
\begin{equation*}
\lbm 2 &1\\ \ell_2-2 &\ell_1+\ell_2+\ell_2-1 \rem.
\end{equation*}

\noindent
Proceeding in this manner yields that $V^\Syl(2;\lambda_2-1)$ has type
\begin{align*}
	\lbm 2 &1\\ 1 &\ell_1 + \SL_{j=1}^{\ell_2-1} (\ell_2-j+1) \rem,
\end{align*}
	
\noindent
and we note that
\begin{equation*}
	\SL_{j=1}^{\ell_2-1} (\ell_2-j+1) = \frac{1}{2}\lp \ell_2-1 \rp \lp \ell_2+2 \rp.
\end{equation*}

\noindent
From Lemma \ref{lem:The Reduction Lemma} and Definition \ref{def:Sylvester Reductions}, we see that
\begin{equation*}
	g(V) = g\lp V^\Syl(2;\lambda_2-1) \rp = 1 + \ell_1 + \frac{1}{2}\lp \ell_2-1 \rp \lp \ell_2+2 \rp.
	\qedhere
\end{equation*}

\end{proof}

\begin{lemma}\label{lem:Obliterating Cubics} \textbf{(Obliterating Cubics)}\\
	Consider an intersection of hypersurfaces $V$ of type $\lbm 3 &2 &1\\ \ell_3 &\ell_2 &\ell_1 \rem$. Then, $V_1^\Syl$ is of type $\lbm 2 &1\\ \beta_3 &\alpha_3 \rem$, where
\begin{align*}
	\beta_3 &= \ell_2 + \frac{1}{2}(\ell_3-1)\ell_3,\\
	\alpha_3 &= \ell_1 + \ell_2\ell_3 + \frac{1}{2}\ell_3(\ell_3+1) + \frac{1}{6}\ell_3\lp 2\ell_3^2-3\ell_3+1 \rp.
\end{align*}	
\end{lemma}

\begin{proof}
	An argument analogous to the proof of Lemma \ref{lem:Obliterating Quadrics} yields that
	\begin{align*}
	\beta_3 &= \ell_2 + \SL_{j=1}^{\ell_3} (\ell_3-j) = \ell_2 + \frac{1}{2}(\ell_3-1)\ell_3.
	\end{align*}
	
	\noindent
	Next, observe that $V^\Syl(3;j)$ has type
	\begin{equation*}
		\lbm 3 &2 &1\\ \ell_3-j &\ell_2 + \SL_{k=1}^j (\ell_3-k) &\lambda_j \rem.
	\end{equation*}	
	
	\noindent
	Consequently,
	\begin{align*}
		\lambda_{j+1} = \lambda_j + \lp \ell_3-j-1 \rp + \lp \ell_2 + \SL_{k=1}^j (\ell_3-k) \rp + 1.
	\end{align*}

	\noindent
	Combined with the initial condition $\lambda_0 = \ell_1$, we obtain that
	\begin{align*}
		\alpha_3 &= \ell_1 + \lp \SL_{j_1=1}^{\ell_3} (\ell_3-j_1+1) \rp + \lp \SL_{j_2=1}^{\ell_3} \ell_2 + \SL_{j_3=2}^{\ell_3} \SL_{j_4=1}^{j_4-1} \ell_3-j_4 \rp,\\
		&= \ell_1 + \frac{1}{2}\ell_3(\ell_3+1) + \lp \ell_2\ell_3 + \SL_{j_3=2}^{\ell_3} \SL_{j_4=1}^{j_3-1} \ell_3-j_4 \rp,\\
		&= \ell_1 + \ell_2\ell_3 + \frac{1}{2}\ell_3(\ell_3+1) + \SL_{j_3=2}^{\ell_3} \SL_{j_4=1}^{j_3-1} (\ell_3-j_2),\\
		&= \ell_1 + \ell_2\ell_3 + \frac{1}{2}\ell_3(\ell_3+1) + \frac{1}{6}\ell_3\lp 2\ell_3^2-3\ell_3+1 \rp.
		\qedhere
	\end{align*}
\end{proof}

\begin{lemma}\label{lem:Obliterating Quartics} \textbf{(Obliterating Quartics)}\\
	Consider an intersection of hypersurfaces $V \subseteq \PP_K^r$ of type $\lbm 4 &3 &2 &1\\ \ell_4 &\ell_3 &\ell_2 &\ell_1 \rem$. Then, $V_1^\Syl$ is of type $\lbm 3 &2 &1\\ \gamma_4 &\beta_4 &\alpha_4 \rem$, where
\begin{align*}
	\gamma_4 &= \ell_3 + \frac{1}{2}(\ell_4-1)\ell_4,\\
	\beta_4 &= \ell_2 + \ell_3\ell_4 + \frac{1}{2}(\ell_4-1)\ell_4 + \frac{1}{6}\ell_4\lp 2\ell_4^2-3\ell_4+1 \rp,\\
	\alpha_4 &= \ell_1 + \ell_4\lp \ell_2+\ell_3+\frac{1}{2}(\ell_4+1) \rp + \ell_4\lp \frac{1}{2}\ell_3(\ell_4+1) + \frac{1}{3}\lp 2\ell_4^2-3\ell_4+1 \rp \rp\\
	& + \frac{1}{24} (\ell_4-2)(\ell_4-1)\ell_4(3\ell_4-1).
\end{align*}
\end{lemma}

\begin{proof}
	The proofs of Lemmata \ref{lem:Obliterating Quadrics} and \ref{lem:Obliterating Cubics} generalize to determine $\gamma_4$ and $\beta_4$ in a straightforward manner. It remains to determine $\alpha_4$. Note that $V^\Syl(4;j)$ has type
\begin{equation*}
	\lbm 4 &3 &2 &1\\ \ell_4-j &\ell_3 + \SL_{k_1=1}^j (\ell_4-k_1) &\ell_2 + \lp \SL_{k_2=1}^j \ell_4 - k_2 \rp + \SL_{k_3=1}^j \lp \ell_3 + \SL_{k_4=1}^{j-1} (\ell_4-k_4) \rp &\lambda_j \rem.
\end{equation*}

\noindent
As a result,
\begin{align*}
	\lambda_{j+1} = \lambda_j + (\ell_4-j-1) + \lp \ell_3 + \SL_{k=1}^j (\ell_4-k) \rp + \lp \ell_2 + \lp \SL_{k_1=1}^j \ell_4 - k_1 \rp + \SL_{k_2=1}^j \lp \ell_3 + \SL_{k_3=1}^{j-1} (\ell_4-k_3) \rp \rp + 1.
\end{align*}
	
	\noindent
	Given the initial condition $\lambda_0 = \ell_1$, it follows that
	\begin{align*}
		\alpha_4 &= \ell_1 + \lp \SL_{j_1=1}^{\ell_4} \ell_4-j_1+1 \rp + \lp \SL_{j_2=1}^{\ell_4} \ell_3 + \SL_{j_3=2}^{\ell_4} \SL_{j_4=1}^{j_3-1} (\ell_4-j_4) \rp\\
		&+ \lp \SL_{j_5=1}^{\ell_4} \ell_2 + \SL_{j_6=2}^{\ell_4} \SL_{j_7=1}^{j_6-1} (\ell_4-j_7) + \SL_{j_8=2}^{\ell_4} \SL_{j_9=1}^{j_8-1} \ell_3 + \SL_{j_{10}=3}^{\ell_4} \SL_{j_{11}=2}^{j_{10}-1} \SL_{j_{12}=1}^{j_{11}-1} (\ell_4-j_{12}) \rp,\\
		\\
		&= \ell_1 + \lp \frac{1}{2}\ell_4(\ell_4+1) \rp + \lp \ell_3\ell_4 + \frac{1}{6}\ell_4\lp 2\ell_4^2-3\ell_4+1 \rp \rp\\
		&+ \lp \ell_2\ell_4 + \frac{1}{6}\ell_4\lp 2\ell_4^2-3\ell_4+1 \rp + \frac{1}{2}(\ell_4-1)\ell_4\ell_3 + \SL_{j_{10}=3}^{\ell_4} \SL_{j_{11}=2}^{j_{10}-1} \SL_{j_{12}=1}^{j_{11}-1} (\ell_4-j_{12}) \rp,\\
		\\
		&= \ell_1 + \ell_4\lp \ell_2+\ell_3+\frac{1}{2}(\ell_4+1) \rp + \ell_4\lp  \frac{1}{2}\ell_3(\ell_4-1) + \frac{1}{3}\lp 2\ell_4^2-3\ell_4+1 \rp \rp + \SL_{j_{10}=3}^{\ell_4} \SL_{j_{11}=2}^{j_{10}-1} \SL_{j_{12}=1}^{j_{11}-1} (\ell_4-j_{12}),\\
		&= \ell_1 + \ell_4\lp \ell_2+\ell_3+\frac{1}{2}(\ell_4+1) \rp + \ell_4\lp \frac{1}{2}\ell_3(\ell_4+1) + \frac{1}{3}\lp 2\ell_4^2-3\ell_4+1 \rp \rp + \frac{1}{24} (\ell_4-2)(\ell_4-1)\ell_4(3\ell_4-1).
		\qedhere
	\end{align*}
\end{proof}

\newpage

\subsection{Appendix C: The Geometric Obliteration Algorithm with Computational Improvements}\label{subsec:Appendix C - The Geometric Obliteration Algorithm with Computational Improvements}

\begin{algorithm}

\begin{algorithmname}\label{alg:The Geometric Obliteration Algorithm with Computational Improvements} \textbf{(The Geometric Obliteration Algorithm with Computational Improvements)}
\end{algorithmname}

\vspace{-12pt}

\hrulefill

\vspace{-12pt}

\begin{itemize}
	\item Input: An intersection of hypersurfaces $V$ of type $\lbm d &d-1 &\cdots &2 &1\\ \ell_d &\ell_{d-1} &\cdots &\ell_2 &\ell_1 \rem$ with $d \geq 2$, encoded as the list $\text{DegreeList} = [\ell_d, \ell_{d-1}, \dotsc, \ell_2, \ell_1]$.
	\item Output: The geometric dimension bound $g(d;\ell_d,\dotsc,\ell_1)$. 
\end{itemize}

\vspace{-12pt}

\hrulefill

\begin{algorithmic}[1]
	
	\Statex We will use the same functions \textproc{ComputePolarCone} and \textproc{ObliterateLargestDegreeHypersurfaces} which were originally defined in Algorithm \ref{alg:The Geometric Obliteration Algorithm}.
	
	\Statex
	
	\Statex We now implement Lemma \ref{lem:Obliterating Quartics} (respectively, Lemmata \ref{lem:Obliterating Cubics} and \ref{lem:Obliterating Quadrics}) via the following three functions.
	
	\Statex
	
	\Function{ObliterateQuartics}{List}:
		\State a = List[0]
		\State b = List[1]
		\State c = List[2]
		\State d = List[3]
		\State gammafour = b + (1/2)*(a-1)*a
		\State betafour = c + a*b + (1/2)*a*(a+1) + (1/6)*(a-1)*a*(2*a-1)
		\State alphafour = d + a*(b+c+(1/2)*(a+1)) + a*((1/2)*b*(a-1)+(1/3)*((2*(a**2))-(3*a)+1))
		\State \hspace{0.7in} + (1/24)*(a-2)*(a-1)*a*(3*a-1)
		\State \Return [gammafour,betafour,alphafour]	
	\EndFunction
	
	\Statex

	\Function{ObliterateCubics}{List}:
		\State a = List[0]
		\State b = List[1]
		\State c = List[2]
		\State betathree = b + (1/2)*(a-1)*a
		\State alphathree = c + a*b + (1/2)*a*(a+1) + (1/6)*a*((2*(a**2))-(3*a)+1)
		\State \Return [betathree,alphathree]
	\EndFunction

	\Statex
			
	\Function{ObliterateQuadrics}{List}:
		\State a = List[0]
		\State b = List[1]
		\State alphatwo = b + (1/2)*a*(a+1)
		\State \Return [1,alphatwo]
	\EndFunction
	
\algstore{bkbreak}
\end{algorithmic}
\end{algorithm}

\begin{algorithm}
\begin{algorithmic}[1]
\algrestore{bkbreak}

	\Statex The \textproc{Main} procedure works very similarly to its counterpart in Algorithm \ref{alg:The Geometric Obliteration Algorithm}, with the only differences being the use of specialized functions to obliterate quartic, cubic, and quadric hypersurfaces.
	
	\Statex
	
	\Procedure{Main}{DegreeList}:
		\If{\Len(DegreeList) == 2:}
			\State FinalDegreeList = \textproc{ObliterateQuadrics}(DegreeList)
			\State Sum = FinalDegreeList[0] = FinalDegreeList[1]
			\State \Return Sum
		\ElsIf{\Len(DegreeList) == 3:}
			\State TempDegreeList = \textproc{ObliterateCubics}(DegreeList)
			\State DegreeList = TempDegreeList
			\State TempDegreeList = \textproc{ObliterateQuadrics}(DegreeList)
			\State FinalDegreeList = TempDegreeList
			\State Sum = FinalDegreeList[0] = FinalDegreeList[1]
			\State \Return Sum
		\ElsIf{\Len(DegreeList == 4:}
			\State TempDegreeList = \textproc{ObliterateQuartics}(DegreeList)
			\State DegreeList = TempDegreeList
			\State TempDegreeList = \textproc{ObliterateCubics}(DegreeList)
			\State DegreeList = TempDegreeList
			\State TempDegreeList = \textproc{ObliterateQuadrics}(DegreeList)
			\State FinalDegreeList = TempDegreeList
			\State Sum = FinalDegreeList[0] = FinalDegreeList[1]
			\State \Return Sum
		\Else:
			\For{index \In \ \Range(1,\Len(DegreeList)-3):}
				\State TempDegreeList = \textproc{ObliterateLargestDegreeHypersurfaces}(DegreeList)
				\State DegreeList = TempDegreeList
			\EndFor
			\State TempDegreeList = \textproc{ObliterateQuartics}(DegreeList)
			\State DegreeList = TempDegreeList
			\State TempDegreeList = \textproc{ObliterateCubics}(DegreeList)
			\State DegreeList = TempDegreeList
			\State TempDegreeList = \textproc{ObliterateQuadrics}(DegreeList)
			\State FinalDegreeList = TempDegreeList
			\State Sum = FinalDegreeList[0] = FinalDegreeList[1]
			\State \Return Sum
		\EndIf
	\EndProcedure
\end{algorithmic}

\end{algorithm}

\newpage

\subsection{Appendix D: The Geometric Obliteration Algorithm for $\mcc^{m-d-1}(\tau_{1,\dotsc,d};P_0,\dotsc,P_{m-d-1})$}\label{subsec:Appendix D - The Geometric Obliteration Algorithm Applied to Polar Cones of Tschirnhaus Complete Intersections}

\begin{algorithm}

\begin{algorithmname}\label{alg:The Geometric Obliteration Algorithm Applied to Tschirnhaus Complete Intersections} \textbf{(The Geometric Obliteration Algorithm for $\mcc^{m-d-1}(\tau_{1,\dotsc,d};P_0,\dotsc,P_{m-d-1})$)}
\end{algorithmname}

\vspace{-12pt}

\hrulefill

\vspace{-12pt}

\begin{itemize}
	\item Imported Packages: scipy.special, math
	\item Input: A positive integer $d$ and and another positive integer $m \geq d+2$.
	\item Output: The optimal reduction bound of $\tau_{1,\dotsc,d}$ for $m$, $\Xi(m,d)$. 
\end{itemize}

\vspace{-12pt}

\hrulefill

\begin{algorithmic}[1]
	\Statex We will use the same functions \textproc{ComputePolarCone} and \textproc{ObliterateLargestDegreeHypersurfaces} which were originally defined in Algorithm \ref{alg:The Geometric Obliteration Algorithm},  as well as the functions \textproc{ObliterateQuartics} and \textproc{ObliterateCubics} which originally defined in Algorithm \ref{alg:The Geometric Obliteration Algorithm with Computational Improvements}.
	
	\Statex
	
	\Statex We first implement a closed form for the type of an $(m-d-1)^{st}$ polar cone of $\tau_{1,\dotsc,d}$, which is Proposition 2.26 of \cite{Sutherland2021C}.
	
	\Statex
	
	\Function{PolarConeOfTschirnhausType}{Type,Level}:
	\State ReturnList = [1]
	\For{counter \In \ \Range(1,Type):}
		\State NewTerm = scipy.special.comb((Level+counter), counter, exact=True)
		\State OutputList.append(NewTerm)
	\EndFor
	\State \Return ReturnList
	\EndFunction
	
\algstore{bkbreak}
\end{algorithmic}
\end{algorithm}

\begin{algorithm}
\begin{algorithmic}[1]
\algrestore{bkbreak}
	
	\Statex This function takes the type of an $(m-d-1)^{st}$ polar cone of $\tau_{1,\dotsc,d}$ as an input and outputs $\Xi(m,d)$.
	
	\Statex
	
	\Function{ObliterateAMinimalNumberOfQuadrics}{List}:
	\State a = List[0]
	\State b = List[1]
	\State Dimension = b + (1/2)*(a**2 + a - 2)	
	\State NumberOfQuadrics = 1	
	\State DimensionList = [Dimension]
	\While{ 2**NumberOfQuadrics $<$  Dimension:}
		\State NumberOfQuadrics += 1
		\State Dimension = NumberOfQuadrics 
		\State           + (1/2)*(a**2 + a - NumberOfQuadrics**2 - NumberOfQuadrics)
		\State DimensionList.append(Dimension)
	\EndWhile
	\State MaxList1 = [2**(NumberOfQuadrics-1)+1, DimensionList[NumberOfQuadrics-2]+m-d+1]
	\State MaxList2 = [2**NumberOfQuadrics+1, DimensionList[NumberOfQuadrics-1]+m-d+1]
	\State Max1 = max(MaxList1[0], MaxList1[1])
	\State Max2 = max(MaxList2[0], MaxList2[1])
	\If{ Max2 $<$ Max1:}
		\If{ MaxList2[1] $<$ MaxList2[0]: }
			\State \Return MaxList2[0]
		\Else:
			\State \Return MaxList2[1]
		\EndIf
	\Else:
		\If{ MaxList1[1] $<$ MaxList1[0]: }
			\State \Return MaxList1[0]
		\Else:
			\State \Return MaxList1[1]
		\EndIf
	\EndIf
	\EndFunction

\algstore{bkbreak}
\end{algorithmic}
\end{algorithm}

\begin{algorithm}
\begin{algorithmic}[1]
\algrestore{bkbreak}

	\Statex The \textproc{Main} procedure functions similarly to its counterpart in Algorithm \ref{alg:The Geometric Obliteration Algorithm with Computational Improvements}. The two differences are that the degree list is computed based on $m$ and $d$ and the use of \textproc{ObliterateAMinimalNumberQuadrics} instead of \textproc{ObliterateQuadrics}.
	
	\Statex
	
	\Procedure{Main}{m,d}:
		\State PolarConeLevel = m-d-1
		\State DegreeList = \textproc{PolarConeOfTschirnhausType}(d,PolarConeLevel)
		\If{\Len(DegreeList) == 2:}
			\State \Return \textproc{ObliterateAMinimalNumberQuadrics}(DegreeList)
		\ElsIf{\Len(DegreeList) == 3:}
			\State TempDegreeList = \textproc{ObliterateCubics}(DegreeList)
			\State DegreeList = TempDegreeList
			\State \Return \textproc{ObliterateAMinimalNumberQuadrics}(DegreeList)
		\ElsIf{\Len(DegreeList == 4:}
			\State TempDegreeList = \textproc{ObliterateQuartics}(DegreeList)
			\State DegreeList = TempDegreeList
			\State TempDegreeList = \textproc{ObliterateCubics}(DegreeList)
			\State DegreeList = TempDegreeList
			\State \Return \textproc{ObliterateAMinimalNumberQuadrics}(DegreeList)
		\Else:
			\For{index \In \ \Range(1,\Len(DegreeList)-3):}
				\State TempDegreeList = \textproc{ObliterateLargestDegreeHypersurfaces}(DegreeList)
				\State DegreeList = TempDegreeList
			\EndFor
			\State TempDegreeList = \textproc{ObliterateQuartics}(DegreeList)
			\State DegreeList = TempDegreeList
			\State TempDegreeList = \textproc{ObliterateCubics}(DegreeList)
			\State DegreeList = TempDegreeList
			\State \Return \textproc{ObliterateAMinimalNumberQuadrics}(DegreeList)
		\EndIf
	\EndProcedure
\end{algorithmic}
\end{algorithm}

\vspace{12pt}

\noindent
\emph{Curtis Heberle}\\
\emph{curtis.heberle@tufts.edu}\\

\noindent
Department of Mathematics\\
Tufts University\\
503 Boston Avenue\\
Bromfield-Pearson\\
Medford, MA 02155\\

\noindent
\emph{Alexander J. Sutherland (corresponding author)}\\
\emph{asuther1@uci.edu}\\

\noindent
340 Rowland Hall\\
Department of Mathematics\\
University of California, Irvine\\
Irvine, CA 92697\\

\noindent
\textbf{Mathematics Subject Classification:} 14G25 (Primary); 12E12, 13F20 (Secondary)\\

\noindent
\textbf{Key Words:} Resolvent degree, polynomials, rational points



\newpage

\begin{thebibliography}{9}

\bibitem[AS1976]{ArnoldShimura1976}
V.I. Arnol'd and G. Shimura, \emph{Superpositions of algebraic functions,} Proc. Symposia in Pure Math, AMS, Providence, 28:45-46, 1976.

\bibitem[Ber1923]{Bertini1923}
E. Bertini, \emph{Introduzione alla geometria projettiva degli iperspazi con appendice sulle curve algebriche e loro singolarit\`a. Seconda edizione riveduta ed ampliata}. Messina, G. Principato, 1923.

\bibitem[Bra1975]{Brauer1975}
R. Brauer, \emph{On the resolvent problem,} Ann. Mat. Pura Appl., (4) 102:45-55, 1975.



\bibitem[Bri1786]{Bring1786}
E. Bring, \emph{Meletemata qu{\ae}dam Mathematica circa Transformationem {\AE}quationum Algebraicarum (``Some Selected Mathematics on the Transformation of Algebraic Equations''),} Lund, 1786.



\bibitem[Che1954]{Chebotarev1954}
G.N. Chebotarev, \emph{On the problem of resolvents.,} Kazan. Gos. Univ. U\v{c}. Zap., (2) 114:189-193, 1954.


\bibitem[CHM2017]{ChenHeMcKay2017}
A. Chen, Y-H. He, and J. McKay, \emph{Erland Samuel Bring's ``Transformation of Algebraic Equations,''} 2017, arXiv:1711.09253v1.


\bibitem[DM1998]{DebarreManivel1998}
O. Debarre and L. Manivel, \emph{Sur la vari\'et\'e des espaces lin\'eaires contenus dans une intersection compl\`ete,} Math. Ann., 312(3):549-574, 1998.


\bibitem[Dix1993]{Dixmier1993}
J. Dixmier, \emph{Histoire de 13e probl\`{e}me de Hilbert,} Cahiers du s\'{e}minare d'histoire des math\'{e}matiques, 3(2):85-94, 1993.


\bibitem[Dol2012]{Dolgachev2012}
I. Dolgachev, \emph{Classical Algebraic Geometry: A Modern View,} Cambridge: Cambridge University Press, 2012.


\bibitem[FW2019]{FarbWolfson2019}
B. Farb and J. Wolfson, \emph{Resolvent degree, {H}ilbert's 13th problem and geometry,} Enseign. Math., 65(3-4):303-376, 2019.


\bibitem[Ham1836]{Hamilton1836}
W. Hamilton, \emph{Inquiry into the validity of a method recently proposed by George B. Jerrard, esq., for transforming and resolving equation of elevated degrees,} Report of the Sixth Meeting of the British Assocation for the Advancement of Science, 295-348, 1836.

\bibitem[Har2010]{Harris2010}
J. Harris, \emph{Algebraic Geometry,} Springer: New York, 2010.

\bibitem[Heb2021]{Heberle2021}
C. Heberle, \emph{Tschirnhaus Transformations, Resolvent Degree, and Sylvester's Method of Obliteration}, in preparation.

\bibitem[Hil1927]{Hilbert1927}
D. Hilbert, \emph{\"{U}ber die {G}leichung neunten {G}rades,} Math. Ann., 97(1):243-250, 1927.

\bibitem[Kle1884]{Klein1884}
F. Klein, \emph{Vorlesungen \"uber das Ikosaeder und die Aufl\"osung der Gleichungen vom f\"unften Grade,} Teubner, Leipzig, 1884.


\bibitem[Kle1887]{Klein1887}
F. Klein, \emph{Zur Theorie der allgemeinen Gleichungen sechsten und siebenten Grades,} Math. Ann., 28 (4):499-532, 1887.




\bibitem[Kle1905]{Klein1905}
F. Klein, \emph{\"Uber die Aufl\"osung der allgemeinen Gleichungen f\"unften und sechsten Grades,} J. Reine Angew. Math., 129:150-174, 1905.


\bibitem[Mor1956]{Morrice1956}
G.G. Morrice, \emph{Felix Klein's ``Lectures on the icosahedron and solution of equation of fifth degree,''} 2nd and rev. edition, New York, Dover Publications, 1956.

\bibitem[Ost1959]{Ostrowski1959}
A.M. Ostrowski, \emph{A Quantitative Formulation of Sylvester's Law of Inertia,} Proc. Natl. Acad. Sci. USA, 45:740-744, 1959.


\bibitem[Rei2022]{Reichstein2022}
Z. Reichstein, \emph{Hilbert's 13th Problem for Algebraic Groups}, 2022, arXiv:2204.13202.

\bibitem[Rob1955]{Robbins1955}
H. Robbins, \emph{A Remark on Stirling's Formula,} Amer. Math. Monthly, 62(1):26, 1955.

\bibitem[Seg1945]{Segre1945}
B. Segre, \emph{The Algebraic Equations of Degrees {$5$}, {$9$}, {$157\dotsc$}, and the Arithmetic Upon an Algebraic Variety,} Ann. of Math., 46(2):287-301, 1945.

\bibitem[Sut2019]{Sutherland2019}
A. Sutherland, \emph{Felix Klein's ``About the Solution of General Equations of Fifth and Sixth Degree (Excerpt from a letter to Mr. K. Hensel),''} 2019, arXiv:1911.02358.

\bibitem[Sut2021A]{Sutherland2021A}
A. Sutherland, \emph{Anders Wiman's ``On the Application of Tschirnhaus Transformations to the Reduction of Algebraic Equations,''} 2021, arXiv:2106.09247.

\bibitem[Sut2021B]{Sutherland2021B}
A. Sutherland, \emph{G. N. Chebotarev's ``On the Problem of Resolvents,''} 2021, arXiv:2107.01006.

\bibitem[Sut2021C]{Sutherland2021C}
A. Sutherland, \emph{Upper Bounds on Resolvent Degree and Its Growth Rate,} 2021, arXiv:2107.08139 

\bibitem[Syl1887]{Sylvester1887}
J.J. Sylvester, \emph{On the so-called Tschirnhausen Transformation,} J. Reine Angew. Math., 100:465-486, 1887.


\bibitem[SH1887]{SylvesterHammond1887}
J.J. Sylvester and J. Hammond, \emph{On Hamilton's numbers,} Philos. Trans. R. Soc. Lond., A, 178:285-312, 1887.


\bibitem[SH1888]{SylvesterHammond1888}
J.J. Sylvester and J. Hammond, \emph{On Hamilton's numbers, Part II,} Philos. Trans. R. Soc. Lond., A, 179:65-72, 1888.


\bibitem[Tsc1683]{Tschirnhaus1683}
E. von Tschirnhaus, \emph{Methodus auferendi omnes terminos intermedios ex data aeqvatione (Method of eliminating all intermediate terms from a given equation),} Acta Eruditorum, 204-207, 1683.


\bibitem[Wal2008]{Waldron2008}
A. Waldron, \emph{Fano Varieties of Low-Degree Smooth Hypersurfaces and Unirationality,} Bachelor thesis, Harvard University, Cambridge, Massachusetts, 2008.


\bibitem[Wim1927]{Wiman1927}
A. Wiman, \emph{\"Uber die Anwendung der Tschirnhausen-Transformation auf die Reduktion algebraischer Gleichungen,} Nova Acta R. Soc. scient. Uppsala, 4(16), 1927.


\bibitem[Wol2021]{Wolfson2021}
J. Wolfson, \emph{Tschirnhaus transformations after Hilbert,} Enseign. Math., 66(3):489-540, 2021.

\end{thebibliography}
\end{document}